\RequirePackage{etoolbox}
\csdef{input@path}{{style/}{graphics/}}
\documentclass[aap,MSNbibl,seceqn,dvips]{arximspdf}
\makeatletter
   \@ifpackageloaded{graphicx}{}{\usepackage{graphicx}}
\makeatother
\usepackage{dcolumn}
\usepackage{accents}
\usepackage{graphicx}

\doi{10.1214/14-AAP1030} 
\volume{25}
\issue{3}
\pubyear{2015}
\firstpage{1540}
\lastpage{1580}

\makeatletter
\newcommand{\widebar}{\overline}
\newcolumntype{d}[1]{D{.}{.}{#1}}
\newcommand{\SIN}{\operatorname{SIN}}
\newcommand{\COS}{\operatorname{COS}}
\newcommand{\unde}{\underaccent{\bar}}
\newcommand{\T}{T}

\newtheorem{teo}{Theorem}[section]
\newtheorem{lem}[teo]{Lemma}
\newtheorem{prop}[teo]{Proposition}
\newproclaim{rem}[teo]{Remark}
\newproclaim{eg}[teo]{Example}
\newproclaim{assum}[teo]{Assumption}
\makeatother

\begin{document}
\begin{frontmatter}

\title{A monotone scheme for high-dimensional fully~nonlinear PDEs}
\runtitle{Monotone scheme for high-dimensional PDE}

\begin{aug}
\author[A]{\fnms{Wenjie} \snm{Guo}\thanksref{T1}\ead[label=e1]{09110180004@fudan.edu.cn}},
\author[B]{\fnms{Jianfeng} \snm{Zhang}\corref{}\thanksref{T2}\ead[label=e2]{jianfenz@usc.edu}}
\and
\author[B]{\fnms{Jia} \snm{Zhuo}\ead[label=e3]{jiazhuo@usc.edu}}
\runauthor{W. Guo, J. Zhang and J. Zhuo}
\affiliation{Fudan University, University of Southern California\\ and
University of Southern California}
\address[A]{W. Guo\\
School of Mathematics \\
Fudan University \\
Shanghai 200433\\
China P.R.\\
\printead{e1}} 
\address[B]{J. Zhang\\
J. Zhuo\\
Department of Mathematics\\
University of Southern California\\
Los Angeles, California 90089-2532\\
USA\\
\printead{e2}\\
\phantom{E-mail: }\printead*{e3}}
\end{aug}
\thankstext{T1}{Part of the research for this paper was done when this
author was visiting the University of Southern California, whose
hospitality is greatly appreciated.}
\thankstext{T2}{Supported in part by NSF Grant DMS-10-08873.}

\received{\smonth{2} \syear{2013}}
\revised{\smonth{3} \syear{2014}}

%
\begin{abstract}
In this paper we propose a feasible numerical scheme for
high-dimen\-sional, fully nonlinear parabolic PDEs, which includes the
quasi-linear PDE associated with a coupled FBSDE as a special case. Our
paper is strongly motivated by the remarkable work Fahim, Touzi and
Warin [\textit{Ann. Appl. Probab.} \textbf{21} (2011) 1322--1364]
and stays in the paradigm of monotone schemes initiated by Barles and
Souganidis [\textit{Asymptot. Anal.} \textbf{4} (1991) 271--283]. Our
scheme weakens a critical constraint imposed by Fahim, Touzi and Warin
(2011), especially when the generator of the PDE depends only on the
diagonal terms of the Hessian matrix. Several numerical examples, up to
dimension~12, are reported.
\end{abstract}

%
\end{frontmatter}

\setcounter{footnote}{2}
\section{Introduction}\label{sec1}

In this paper we are interested in feasible numerical schemes for the
following fully nonlinear parabolic PDE on $[0, T]\times\mathbb
{R}^d$, especially in high-dimensional cases,
%
%
\begin{equation}
\label{PDE} -\partial_t u -G \bigl(t,x,u,Du,D^2u
\bigr)=0;\qquad u(T,\cdot)=g(\cdot).
\end{equation}
%
The standard numerical schemes in the PDE literature, for example,
finite difference methods and finite elements methods, work only for
low-dimensional \mbox{problems}, typically $d \le3$, due to the well-known
\emph{curse of dimensionality}. However, in many applications,
especially in finance, the dimension $d$ can be higher. We thus turn to
the probabilistic approach, which is less sensitive to the dimension.

In the semilinear case $G = \frac{1}{2}\operatorname{tr}[\sigma
^2(t,x) D^2 u] +
f(t,x, u, D u)$,
PDE (\ref{PDE}) is associated to a Markovian backward SDE due to the
nonlinear Feynman--Kac formula introduced by Pardoux and Peng \cite
{PP}. Based on the regularity results of \mbox{BSDEs} established by Zhang
\cite{Zhang}, Bouchard and Touzi \cite{BT} and Zhang \cite{Zhang}
proposed the so called backward Euler scheme for such BSDEs and hence
for the semilinear PDEs, and obtained the rate of convergence. This
scheme approximates the BSDE by a sequence of conditional expectations,
and several efficient numerical algorithms have been proposed to
compute these conditional expectations, notably: Bouchard and Touzi
\cite{BT}, Gobet, Lemor and Warin \cite{GLW}, Bally, Pages and
Printems \cite{BPP}, Bender and Denk \cite{BD} and Crisan and
Manolarakis \cite{CM}. There have been numerous publications on the
subject, and the schemes have been extended to more general BSDEs, for
example, reflected BSDEs which correspond to obstacle PDEs and are
appropriate for pricing and hedging American options. Typically these
algorithms work for 10 or even higher-dimensional problems.

We intend to numerically solve PDE (\ref{PDE}) in the fully nonlinear
case, in particular the Hamilton--Jacobi--Bellman equations and the
Bellman--Isaacs equations which are widely used in stochastic control
and in stochastic differential games. We remark that this is actually
one main motivation of the developments of second order BSDEs
by Cheridito et al. \cite{CSTV} and Soner, Touzi and Zhang \cite
{STZ}. Our scheme is strongly inspired by the work of Fahim, Touzi and
Warin \cite{FTW}. Based on the monotone scheme of Barles and
Souganidis \cite{BS}, Fahim, Touzi and Warin~\cite{FTW} extended the backward Euler scheme to fully
nonlinear PDE (\ref{PDE}). In the case that $G$ is convex in $(u, Du,
D^2 u)$, they obtained the rate of convergence by using the techniques
in Krylov \cite{Krylov} and Barles and Jakobsen \cite{BJ}. They
applied the linear regression method (see, e.g., \cite{GLW}), to
compute the involved conditional expectations, and presented some
numerical examples up to dimension 5. We remark that the rate of
convergence has been improved recently by Tan \cite{Tan}, by using
purely probabilistic arguments.

There is one critical constraint in \cite{FTW} though. In order to
ensure the monotonicity of the backward Euler scheme, they assume the
lower and upper bounds of $G_\gamma$, the derivative of $G$ with
respect to $D^2 u$, satisfies certain constraint. However, when the
dimension is high, this constraint implies that $G_\gamma$ is
essentially a constant, and thus PDE (\ref{PDE}) is essentially
semilinear; see (\ref{FTWconstraint}) for more details. This is, of
course, not desirable in practice.

The main contribution of this paper is to propose a new scheme so as to
relax the above constraint. In \cite{FTW} the involved conditional
expectations are expressed in terms of Brownian motion, which is
unbounded. Our first simple but important observation is that we may
replace it with a bounded trinomial tree, which helps to maintain the
monotonicity of the scheme. We next modify the scheme by introducing a
new kernel for the Hessian approximation [see the $K_2(\xi)$ in (\ref
{Dh}) below], but still in the paradigm of monotone scheme. In the
special case where $G_\gamma$ is diagonal, namely $G$ involves $D^2
u$ only through its diagonal terms, the above constraint is removed
completely. Rate of convergence of our scheme is also obtained.
Several numerical examples are presented. In the low-dimensional case,
our scheme is comparable to finite difference method and is superior to
the simulation methods. When $G_\gamma$ is diagonal, our scheme works
well for 12-dimensional problems.

We note that PDE (\ref{PDE}) covers the quasi-linear PDEs as a special case,
which corresponds to a coupled forward--backward SDE due to the
four-step scheme of Ma, Protter and Yong \cite{MPY}. There are only a
few papers on numerical methods for FBSDEs, for example, Douglas, Ma
and Protter \cite{DMP}, Makarov \cite{Makarov}, Cvitanic and Zhang
\cite{CZ}, Delarue and Menozzi \cite{DM}, Milstein and Tretyakov
\cite{MT}, Bender and Zhang \cite{BZ} and Ma, Shen and Zhao \cite
{MSZ}. Most of them deal with low-dimensional FBSDEs only, except that
\cite{BZ} reported a 10-dimensional numerical example. However, \cite
{BZ} proved the rate of convergence only for time discretization, and
the convergence of the linear regression approximation is not analyzed
theoretically. Our scheme works for FBSDEs as well, especially when the
diffusion coefficient $\sigma$ is diagonal. A numerical example for a
12-dimensional coupled FBSDE is reported.

We have also presented a few numerical examples which violate our
assumptions, and thus the scheme may not be monotone. Numerical results
show that our scheme still converges. In particular, we note that our
current theoretical result does not cover the $G$-expectation, a
nonlinear expectation introduced by Peng~\cite{Peng}. We nevertheless
implement our scheme for a 10-dimensional HJB equation, which
corresponds to a second order BSDE and includes the $G$-expectation as
a special case, and it indeed converges to the true solution. It will
be very interesting to investigate the convergence of our scheme, or
its variations if necessary, when the monotonicity condition is
violated. We shall leave this for our future research.

Finally, we note that we have recently extended the idea of monotone
schemes to the so called path dependent PDEs; see Zhang and Zhuo \cite{ZZ}.

The rest of the paper is organized as follows.
In Section~\ref{sec2} we present some preliminaries. In Section~\ref{sec3} we propose our
scheme and prove the main convergence results. Section~\ref{sec4} is devoted to
the study of quasi-linear PDEs and the associated coupled FBSDEs. In
Section~\ref{sec5} we discuss how to approximate the involved conditional
expectations. Finally we present several numerical examples in
Section~\ref{sec6}, up to dimension 12. 

\section{Preliminaries}\label{sec2}
\label{sect-Prelim}

Let $T>0$ be the terminal time, $d\ge1$ the dimension of the state
variable $x$, $\mathbb{S}^d$ the set of $d\times d$ symmetric matrices
and $\mathbb{R}^{d\times d}_+$ the set of nondegenerate $d\times d$
matrices. For $\gamma, \tilde\gamma\in\mathbb{S}^d$, we say
$\gamma\le\tilde\gamma$ if $\tilde\gamma- \gamma$ is
nonnegative definite. For $x, \tilde x \in\mathbb{R}^d$ and $\gamma, \tilde\gamma\in\mathbb{R}^{d\times d}$, denote
\begin{eqnarray*}
x \cdot\tilde x&:=& \sum_{i=1}^d
x_i \tilde x_i,\qquad |x|:= \sqrt{ x \cdot x}\quad
\mbox{and}\quad \gamma\dvtx \tilde\gamma:= \operatorname {tr} \bigl(\gamma\tilde
\gamma^\T \bigr),\qquad |\gamma|:= \sqrt{\gamma\dvtx \gamma},
\end{eqnarray*}
where $^\T$ denotes transpose.
For any $\gamma= [\gamma_{ij}]\in\mathbb{S}^d$, denote
%
%
\begin{equation}
\label{diag} D[\gamma]:= \mbox{the diagonal matrix whose $(i,i)$th component is
$\gamma_{ii}$.}
\end{equation}
It is clear that, for any $\gamma, \tilde\gamma\in\mathbb{S}^d$,
%
%
\begin{equation}
\label{diagproperty} D[\gamma]\dvtx \tilde\gamma= D[\gamma]\dvtx D[\tilde\gamma] =
\gamma\dvtx D[ \tilde\gamma].
\end{equation}
Moreover, we use the same notation ${\mathbf0}$ to denote the zeroes in
$\mathbb{R}^d$ and $\mathbb{S}^d$.

Our objective is PDE (\ref{PDE}), where $G\dvtx  (t,x,y,z,\gamma)\in[0,
T] \times\mathbb{R}^d \times\mathbb{R}\times\mathbb{R}^d \times
\mathbb{S}^d \to\mathbb{R}$ and $g\dvtx  x\in\mathbb{R}^d \to\mathbb
{R}$. We first recall the definition of viscosity solutions: an upper
(resp., lower) semicontinuous function $u$ is called a viscosity
subsolution (resp., viscosity supersolution) of PDE (\ref{PDE}) if
$u(T,\cdot) \le(\mbox{resp.}, \ge)\,  g(x)$ and for any $(t,x) \in
[0, T)\times\mathbb{R}^d$ and any smooth function $\varphi$ satisfying
\[
[u-\varphi](t,x) = 0 \ge(\mbox{resp.}, \le)\, [u-\varphi ](s,y)\qquad\mbox{for all
}(s,y) \in[0, T]\times\mathbb{R}^d,
\]
we have
\[
\bigl[-\partial_t \varphi-G \bigl(\cdot,\varphi,D
\varphi,D^2\varphi \bigr) \bigr](t,x) \le(\mbox{resp.}, \ge)\, 0.
\]
For the theory of viscosity solutions, we refer to the classical
references \cite{CIL} and~\cite{FS}. 
We remark that Barles and Souganidis \cite{BS} consider more general
discontinuous viscosity solutions, which is unnecessary in our
situation due to the regularities we will prove; see also \cite{FTW},
Remark 2.2.
We shall always assume the following standing assumptions:

%
\begin{assum}
\label{assum-standing}
(i) $G(\cdot, 0, {\mathbf0}, {\mathbf0})$ and $g$ are bounded.

(ii) $G$ is continuous in $t$, uniformly Lipschitz continuous in $(x,
y, z, \gamma)$ and $g$ is uniformly Lipschitz continuous in $x$.

(iii) 
PDE (\ref{PDE}) is parabolic; that is, $G$ is nondecreasing in $\gamma$.

(iv) Comparison principle for PDE (\ref{PDE}) holds in the class of
bounded viscosity solutions. That is, if $u_1$ and $u_2$ are bounded
viscosity subsolution and viscosity supersolution of PDE (\ref{PDE}),
respectively, then $u_1 \le u_2$.
\end{assum}

For notational simplicity, throughout the paper we assume
further that
\[
\mbox{$G$ is differentiable in $(y,z,\gamma)$ so that we can
use the notation $ G_\gamma$, etc.}
\]

However, we emphasize that all the results in the paper do
not rely on this additional assumption. Our goal of the paper is to
numerically compute the viscosity solution $u$. In their seminal work
Barles and Souganidis \cite{BS} proposed a monotone scheme in an
abstract way and proved its convergence by using the viscosity solution
approach. To be precise, for any $t\in[0, T)$ and $h \in(0, T-t)$,
let $\mathbb{T}^t_h$ be an operator on the set of measurable functions
$\varphi\dvtx  \mathbb{R}^d \to\mathbb{R}$. For $n\ge1$, denote $h:=
{T\over n}$, $t_i:= i h$, $i=0,1,\ldots, n$, and define
%
%
\begin{equation}
\label{BS-uh} u_h(t_n, \cdot):= g( \cdot),\qquad
u_h(t, \cdot):= \mathbb{T}^t_{t_i
- t}
\bigl[u_h(t_{i},\cdot) \bigr],\qquad t
\in[t_{i-1}, t_i)
\end{equation}
for $i=n,\ldots, 1$. The following convergence result is due to
Fahim, Touzi and Warin \cite{FTW}, Theorem 3.6, which is based on \cite{BS}.

%
\begin{teo}
\label{teo-BS}
Let Assumption~\ref{assum-standing} hold. Assume $\mathbb{T}^t_h$
satisfies the following conditions:
\begin{longlist}[(iii)]
\item[(i)] Consistency: for any $(t,x)\in[0, T)\times\mathbb{R}^d$ and any
$\varphi\in C^{1,2}([0, T)\times\mathbb{R}^d)$,
\begin{eqnarray*}
&&\lim_{(t', x', h, c) \to(t,x,0,0)}{ \mathbb{T}^{t'}_h
 [[c+\varphi](t'+h, \cdot)  ](x') - [c+\varphi](t',x')\over h}
\\
&&\qquad = \partial_t \varphi(t,x) + G \bigl(t,x, \varphi, D \varphi,
D^2\varphi \bigr).
\end{eqnarray*}

\item[(ii)] Monotonicity: $\mathbb{T}^t_h [\varphi] \le\mathbb
{T}^t_h[\psi]$ whenever $\varphi\le\psi$.

\item[(iii)] Stability: $u_h$ is bounded uniformly in $h$ whenever $g$ is bounded.

\item[(iv)] Boundary condition: for any $x\in\mathbb{R}^d$, $\lim_{(t',x',h)\to(T,x,0)}u_h(t',x')=g(x)$.
\end{longlist}
Then the PDE (\ref{PDE}) has a unique bounded viscosity
solution $u$, and $u_h$ converges to $u$ locally uniformly as $h\to0$.
\end{teo}


We\vspace*{2pt} remark that in \cite{FTW} the Monotonicity condition is weakened
slightly. Roughly speaking, Fahim, Touzi and Waxin \cite{FTW} proposed
a scheme $\widebar{\mathbb{T}}^t_h$ as follows. Assume
there exist $\unde \sigma, \bar\sigma\dvtx  [0, T]\times
\mathbb{R}^d\to\mathbb{R}^{d\times d}_+$ such that $ \frac{1}{2}
\unde a(t,x) \le G_\gamma(t,x,y,z,\gamma) \le\frac{1}{2}\bar a(t,x)$, for any $(t,x,y,z,\gamma)$, where $\unde  a:= \unde \sigma \unde \sigma^\T$ and $\bar a:=
\bar\sigma \bar\sigma^\T$. Denote
%
%
\begin{equation}
\label{barF} \widebar F(t,x,y,z,\gamma):= G(t,x,y,z,\gamma) - \tfrac{1}{2}
\unde  a\dvtx \gamma,
\end{equation}
and define
%
%
\begin{equation}
\label{FTWTh} \widebar{\mathbb{T}}^t_h[\varphi](x) := \widebar{\mathcal{D}}^{t,0}_h \varphi(x) + h \widebar F \bigl(t, x, \widebar{\mathcal{D}}^{t,0}_h \varphi (x), \widebar{\mathcal{D}}^{t,1}_h
\varphi(x), \widebar{\mathcal{D}}^{t,2}_h \varphi(x) \bigr),
\end{equation}
where, for a $d$-dimensional standard Normal random variable $N$,
%
%
\begin{eqnarray}
\label{FTWDh}  \widebar{\mathcal{D}}^{t,i}_h \varphi(x)&:=&
\mathbb{E} \bigl[\varphi( x+ \sqrt{h} \unde \sigma N) \widebar K_i(N)
\bigr],\qquad i=0,1,2,\nonumber
\\
\widebar K_0(N)&:=& 1,\qquad\widebar K_1(N):=
{\unde \sigma^{-T}N\over
\sqrt{h}},
\\
\widebar K_2(N)&:=& { \unde \sigma^{-T}[N N^\T - I_d] \unde \sigma^{-1}\over
h},\nonumber
\end{eqnarray}
and $\unde \sigma^{-T}:= (\unde \sigma^{-1})^\T$.
This scheme satisfies the consistency, and the stability follows from
the monotonicity. However, to ensure the monotonicity, one needs to
assume $ \widebar{F}_\gamma\dvtx  \unde  a^{-1} \le1$, see \cite
{FTW}, proof of Lemma 3.12. This essentially requires
%
%
\begin{equation}
\label{FTWconstraint0} \bigl[\tfrac{1}{2} \bar a -
\tfrac{1}{2}\unde  a \bigr]\dvtx \unde  a^{-1} \le1\quad
\mbox{and thus}\quad \bar a\dvtx \unde  a^{-1} \le d+2.
\end{equation}
In the case $\unde  a= \unde \alpha I_d, \bar a=
\bar\alpha I_d$ for some scalar functions $0<\unde \alpha
\le\bar\alpha$, we have
%
%
\begin{equation}
\label{FTWconstraint} 1 \le\bar\alpha\slash\unde \alpha\le1 +2\slash d.
\end{equation}
When $d$ is large, this implies $\bar\alpha\approx\unde
\alpha$, and thus $G$ is essentially semilinear, which of course is
not desirable in practice.

Our goal of this paper is to modify algorithm (\ref{FTWTh})--(\ref
{FTWDh}) so as to relax the above constraint, mainly in the case that
$G_\gamma$ is diagonally dominant. In particular, when $G_\gamma$
is diagonal, we remove this constraint completely. 

\section{The numerical scheme}\label{sec3}
\label{sect-Scheme}

In this section we present our numerical scheme and study its convergence.
Our scheme involves two functions $\sigma_0\dvtx  [0, T]\times\mathbb
{R}^d\to\mathbb{R}^{d\times d}_+$ and $p\dvtx  [0, T]\times\mathbb
{R}^d\to(0,1)$ satisfying
%
%
\begin{equation}
\label{psi0bound} \mbox{$\sigma_0$, $\sigma_0^{-1}$
and $p^{-1}$ are bounded. }
\end{equation}
Denote
%
%
\begin{eqnarray}\label{F}
F(t,x,y,z,\gamma) &:=& G(t,x,y,z,\gamma) - \tfrac{1}{2}
a_0(t,x)\dvtx \gamma
\nonumber
\\
\\[-19pt]
\eqntext{\mbox{where }a_0:= \sigma_0 \sigma_0^\T;}
\\
\widetilde G_\gamma(t,x,y,z,\gamma) &:=& \sigma_0^{-1}(t,x)
G_\gamma (t,x,y,z,\gamma) \sigma_0^{-T}(t,x).
\nonumber
\end{eqnarray}
%
For notational simplicity, we will be suppressing the variables when
there is no confusion. Unlike Fahim, Touzi and Waxin \cite{FTW}, we
emphasize that we do not require $\frac{1}{2}a_0
\le G_\gamma$. Let
$(\Omega, \mathcal{F}, \mathbb{P})$ be a probability space. For each
$(t,x)$, let $\xi:= \xi^{t,x}$ be an $\mathbb{R}^d$-valued random
variable such that its components $\xi_i$, $i=1,\ldots, d$ are
independent and have the identical distribution
%
%
\begin{equation}
\label{xii} \qquad\quad\mathbb{P} \biggl(\xi_i = {1\over \sqrt{p}}
\biggr) = {p\over
 2},\qquad\mathbb {P} \biggl(\xi_i = -
{1\over \sqrt{p}} \biggr) = {p\over
 2},\qquad\mathbb{P}(\xi
_i = 0) = 1-p.
\end{equation}
This implies that
%
%
\begin{equation}
\label{Exii} \mathbb{E}[\xi_i]=\mathbb{E} \bigl[
\xi_i^3 \bigr]=0,\qquad\mathbb{E} \bigl[\xi
_i^2 \bigr]=1,\qquad\mathbb{E} \bigl[
\xi_i^4 \bigr]=\frac{1}{p}.
\end{equation}
%

We now modify algorithm (\ref{FTWTh})--(\ref{FTWDh}),
%
%
\begin{eqnarray}
\label{Th} \mathbb{T}^t_h[\varphi](x) &:=&
\mathcal{D}^{t,0}_h \varphi(x) + h F \bigl(t, x,
\mathcal{D}^{t,0}_h \varphi(x), \mathcal{D}^{t,1}_h
\varphi(x), \mathcal{D}^{t,2}_h \varphi(x) \bigr),
\end{eqnarray}
where, recalling (\ref{diag}) and suppressing the variables $(t,x)$,
%
%
\begin{eqnarray}
\label{Dh}  \mathcal{D}^{t,i}_h \varphi(x)&:=&
\mathbb{E} \bigl[\varphi( x+ \sqrt{h} \sigma_0 \xi) K_i(
\xi) \bigr],\qquad i=0,1,2,\nonumber
\\
K_0(\xi)&:=& 1,\qquad K_1(\xi):={\sigma_0^{-T}\xi\over
 \sqrt{h}},
 \\
K_2(\xi)&:=& {\sigma_0^{-T} [(1-p) \xi\xi^\T - (1-3p) D[\xi\xi
^\T] - 2p I_d ]\sigma_0^{-1} \over
 (1-p)h}.
\nonumber
\end{eqnarray}
One may check straightforwardly that
%
%
\begin{equation}
\label{EK} E \bigl[K_1(\xi) \bigr] = {\mathbf0},\qquad E
\bigl[K_2(\xi) \bigr] = {\mathbf0}.
\end{equation}
We recall that the approximating solution $u_h$ is defined by (\ref{BS-uh}).

%
\begin{rem}
\label{rem-p}
If we assume $\frac{1}{2} a_0 \le G_\gamma$
and set $ p= {1\over
3}$, then our scheme is obtained by replacing the normal random
variable $N$ in (\ref{FTWDh}) with trinomial random variable $\xi$.
This in fact has already been mentioned in \cite{FTW}.
\end{rem}
%

%
\begin{rem}
\label{rem-scheme}
(i) The seemingly complicated kernel $K_2(\xi)$ is to ensure the
consistency of the scheme; see Lemma~\ref{lem-consistency} below.

(ii) The $\sigma_0$ is used to construct the forward process, on
which we will compute the conditional expectations. This is fundamental
in Monte Carlo methods which we will use.

(iii) The introduction of $p$ allows us to obtain the monotonicity of
our scheme; see Section~\ref{sect-mon} below. However, we should point
out that the crucial property is~(\ref{Exii}). Additional freedom of
parameters, for example, by replacing the trinomial tree with
$5$-nomial trees, will not help to weaken the monotonicity condition
Assumption~\ref{assum-mon} below.

(iv) An additional advantage of using trinomial tree (instead of
Brownian motion) is that it is bounded, which helps to ensure the
monotonicity; see the proof of Lemma~\ref{lem-mon} and Remark~\ref
{rem-mon3}(ii) below.
\end{rem}

\subsection{Consistency}\label{sec3.1}
We first justify our scheme by checking its consistency.

%
\begin{lem}
\label{lem-consistency}
Under Assumption~\ref{assum-standing} and (\ref{psi0bound}), $\mathbb
{T}^t_h$ satisfies the consistency requirement in Theorem~\ref{teo-BS}\textup{(i)}.
\end{lem}

\begin{pf} $\!$Fix $(t,x)\in[0, T)\times\mathbb{R}^{d\!}$. Let $\varphi\in
C^{1,2}([0, T]\times\mathbb{R}^d)$ and $(t', x', h, c)\to(t,x,0,
0)$. Apply Taylor expansion to $h$: with the right-hand side taking
values at $(t',x')$,
\begin{eqnarray*}
&&\varphi \bigl(t'+h, x'+ \sqrt{h}
\sigma_0 \xi \bigr)
\\
&&\qquad =\varphi+ \partial_t \varphi h+ \sqrt{h} D \varphi\cdot
\sigma_0 \xi+ {h\over 2} D^2 \varphi\dvtx [
\sigma _0\xi] [\sigma _0\xi]^\T+o(h).
\end{eqnarray*}
We emphasize that, thanks to (\ref{psi0bound}), the $o(\cdot)$ in
this proof is uniform in $(t',x', h,c)$.
By (\ref{Exii}) and the independence of $\xi_k$ one may check
straightforwardly that
%
%
\begin{eqnarray}
\label{D01} \mathcal{D}^{t,0}_h \varphi
\bigl(t'+h,\cdot \bigr) \bigl(x' \bigr) &=& \varphi+
\partial_t \varphi h + {h\over 2} D^{2}
\varphi\dvtx a_0 + o(h);
\nonumber
\\[-8pt]
\\[-8pt]
\mathcal{D}^{t,1}_h \varphi \bigl(t'+h,\cdot
\bigr) \bigl(x' \bigr) &=& D \varphi+ o(\sqrt{h}).
\nonumber
\end{eqnarray}
Moreover, for any $i\neq j$,
\begin{eqnarray*}
\mathbb{E} \bigl[(1-p) \xi\xi^\T - (1-3p) D \bigl[\xi
\xi^\T \bigr] - 2p I_d \bigr] &=& {\mathbf0};
\\
\mathbb{E} \bigl[\xi_i \bigl[(1-p) \xi\xi^\T - (1-3p) D
\bigl[\xi\xi^\T \bigr] - 2p I_d \bigr] \bigr] &=& {
\mathbf0};
\\
\mathbb{E} \bigl[\xi_i^2 \bigl[(1-p) \xi
\xi^\T - (1-3p) D \bigl[\xi\xi ^\T \bigr] - 2p
I_d \bigr] \bigr] &=& 2(1-p) \delta_{i,i};
\\
\mathbb{E} \bigl[\xi_i\xi_j \bigl[(1-p) \xi
\xi^\T - (1-3p) D \bigl[\xi \xi^\T \bigr] - 2p
I_d \bigr] \bigr] &=& (1-p) (\delta_{i,j} +
\delta_{j, i}).
\end{eqnarray*}
Here $\delta_{i, j}$ is the $d\times d$-matrix whose $(i,j)$th
component is $1$, and all other components are $0$. Then, denoting $A
=[a_{i,j}]:= \sigma_0^\T D^2 \varphi\sigma_0$,
\begin{eqnarray*}
&&\mathbb{E} \bigl[ \bigl(D^2 \varphi\dvtx [\sigma_0
\xi] [ \sigma_0\xi ]^\T \bigr) \sigma_0^{-T}
\bigl[(1-p) \xi\xi^\T - (1-3p) D \bigl[\xi\xi^\T \bigr] -
2p I_d \bigr]\sigma_0^{-1} \bigr]
\\
&&\qquad =\sigma_0^{-T} \mathbb{E} \bigl[ \bigl(A\dvtx \xi
\xi^\T \bigr) \bigl[(1-p) \xi\xi^\T - (1-3p)D \bigl[\xi
\xi^\T \bigr] - 2p I_d \bigr] \bigr] \sigma_0^{-1}
\\
&&\qquad = \sigma_0^{-T}\sum_{i, j=1}^d
a_{i,j} \mathbb{E} \bigl[\xi_i\xi _j \bigl[(1-p)
\xi\xi^\T - (1-3p)D \bigl[\xi\xi^\T \bigr] - 2p
I_d \bigr] \bigr] \sigma_0^{-1}
\\
&&\qquad =(1-p)\sigma_0^{-T}\sum_{i, j=1}^d
a_{i,j} (\delta_{i,j} + \delta_{j, i})
\sigma_0^{-1}
\\
&&\qquad = 2(1-p)\sigma_0^{-T} A \sigma_0^{-1}
= 2 (1-p)D^2\varphi,
\end{eqnarray*}
and thus
%
%
\begin{equation}
\label{D2} \mathcal{D}^{t,2}_h \varphi
\bigl(t'+h,\cdot \bigr) \bigl(x' \bigr)=D^2
\varphi+ o(1).
\end{equation}

Plugging (\ref{D01}) and (\ref{D2}) into (\ref{Th}) and recalling
(\ref{EK}),
we have
\begin{eqnarray*}
&& \mathbb{T}^{t'}_h \bigl[[c+\varphi]
\bigl(t'+h, \cdot \bigr) \bigr] \bigl(x' \bigr)
\\
&&\qquad = c+
\varphi+ \partial_t \varphi h + {h\over 2}
D^{2} \varphi\dvtx a_0 + o(h)
\\
&&\quad\qquad{} + h F \biggl(t',
x',c+ \varphi+ \partial_t \varphi h +
{h\over
 2} D^{2} \varphi\dvtx a_0 + o(h),
\\
&&\hspace*{131pt}  D \varphi+ o(\sqrt{h}), D^2\varphi+ o(1) \biggr).
\end{eqnarray*}
Then, by (\ref{F}),
\begin{eqnarray*}
&&{1\over  h} \bigl[\mathbb{T}^{t'}_h
\bigl[[c+\varphi] \bigl(t'+h, \cdot \bigr) \bigr]
\bigl(x' \bigr) - [c+\varphi] \bigl(t',x'
\bigr) \bigr]
\\
&&\qquad = \partial_t \varphi \bigl(t',x'
\bigr) - \frac{1}{2} o(1)\dvtx a_0 \bigl(t',x'
\bigr) + o(1)
\\
&&\quad\qquad{} + G \biggl(t',x', c+ \varphi+
\partial_t \varphi h + {h\over
 2} D^{2}
\varphi\dvtx a_0 + o(h),
\\
&&\hspace*{125pt} D \varphi+ o(\sqrt{h}), D^2
\varphi+ o(1) \biggr).
\end{eqnarray*}
Sending $(t',x',h,c)\to(t,x,0,0)$, we obtain the consistency immediately.
\end{pf}

\subsection{The monotonicity}\label{sec3.2}
\label{sect-mon}
To obtain the monotonicity of our scheme, we need to impose an
additional assumption.
Let $\sigma_0\dvtx  [0, T]\times\mathbb{R}^d \to\mathbb{R}^{d\times
d}_+$ and $\widetilde G_\gamma$, $D[\widetilde G_\gamma]$ be defined by
(\ref{F}) and (\ref{diag}). Introduce the following scalar functions
associated to $\sigma_0$:
%
%
\begin{eqnarray}
\label{th}
\qquad \unde \alpha(t,x)&:=& \sup \bigl\{\alpha>0\dvtx \alpha
I_d \le D[\widetilde G_\gamma](t,x, y,z,\gamma),\
\forall(y,z, \gamma) \bigr\};
\nonumber
\\
\bar\alpha(t,x)&:=& \inf \bigl\{\alpha>0\dvtx \alpha I_d \ge D[
\widetilde G_\gamma](t,x, y,z,\gamma),\ \forall(y,z,\gamma) \bigr\};
\nonumber
\\
\Lambda(t,x) &:=& { \bar\alpha(t,x)\over
 \unde \alpha
(t,x)},\qquad\theta(t,x):= \inf \bigl\{ \theta
\ge0\dvtx D[ \widetilde G_\gamma] \le(1+\theta) \widetilde G_\gamma
\bigr\};
\nonumber\\[-8pt]\\[-8pt]
\alpha_p&:=& {p(2+3\theta)-\theta\over
p(1+\theta)},\qquad c_p:=
\sqrt{2p\Lambda+ \alpha_p - 2p };
\nonumber
\\
\lambda_p &:=& \sqrt{p} \biggl[(1-p + pd) c_p-
{ 2pd\Lambda\over
  c_p
} \biggr] { \sqrt{\unde \alpha}\over |\sigma_0^{-1}|};
\nonumber
\\
\lambda^* &:=& \inf_{(t,x) \in[0, T]\times\mathbb{R}^d} \sup_{p\in[{\theta/( 2(1+\theta))}, {1/3}] \cap(0, {1/3}]}\lambda_p(t,x).
\nonumber
\end{eqnarray}
We remark that if we rescale $\sigma_0$ by a constant $c$, then
$\unde \alpha$ and $\bar\alpha$ will be rescaled by~$c^{-2}$. However, $\Lambda$, $\theta$, $\alpha_p$, $c_p$,
$\lambda_p$ and $\lambda^*$ are all invariant.
The following assumption is crucial.

%
\begin{assum}
\label{assum-mon} There exist $\sigma_0$ and $p$ satisfying (\ref
{psi0bound}) and:
\begin{longlist}[(ii)]
\item[(i)] $\theta(t,x)\leq2$ for all $(t,x)$ and $\lambda^*>0$;\vspace*{2pt}

\item[(ii)] $p\in[{\theta\over 2(1+\theta)}, {1\over
 3}] \cap(0, {1\over
3}]$, $\lambda_p\geq{\lambda^*\over 2}$ and
$\unde {\alpha}=c_p^{-2}$.
%
%
%
%
\end{longlist}
\end{assum}

This assumption is somewhat complicated. We shall provide several
remarks concerning it after proving the monotonicity of our scheme. At
below we first explain our choices of parameters which will be used in
the proof of next lemma.

\begin{rem}
\label{rem-mon1}
(i) We need $p\le{1\over 3}$ so that $1-3p$,
the coefficient of $D[\xi
\xi^\T]$, is nonnegative. For $0\le\theta\le2$, we have ${\theta
\over 2(1+\theta)} \le{1\over
 3}$. Moreover, for $p\in[{\theta
\over 2(1+\theta)},\break {1\over
 3}] \cap(0, {1\over
 3}]$, it holds that
$\alpha_p \ge2p$.\vspace*{1pt} 

(ii) To ensure the monotonicity, we shall first choose $\sigma_0'$
with $|\sigma_0'|=1$ satisfying Assumption~\ref{assum-mon},
preferably the one\vspace*{2pt} maximizing the corresponding $\lambda^*$. We next
choose $p\in[{\theta\over 2(1+\theta)}, {1\over
 3}] \cap(0,
{1\over 3}]$ such that $\lambda_p \ge{\lambda
^*\over 2}$. Finally
we rescale $\sigma_0'$ to obtain $\sigma_0$ satisfying $\unde
\alpha= c_p^{-2}$.

(iii) The above choices of $p$ and $\sigma_0$ is somewhat optimal in
order to maintain the monotonicity. However, given $G$, they may not be
optimal for the convergence of the scheme. For example, a smaller $p$
may help for the monotonicity, but may increase the variance of the
Monte Carlo simulation which will be introduced later. In our numerical
examples in Section~\ref{sect-num} below, we may choose them slightly
differently. It is not clear to us how to choose $p$ and $\sigma_0$
so as to optimize the overall efficiency of the algorithm.
%
%
\end{rem}





%
\begin{lem}
\label{lem-mon}
Let Assumptions~\ref{assum-standing} and~\ref{assum-mon} hold, and
consider the algorithm by using the $p$ and $\sigma_0$ as specified
in Remark~\ref{rem-mon1}\textup{(ii)}. Then there exists a constant \mbox{$h_0>0$},
depending on $d, T,\lambda^*$, and the bounds and Lipschitz constants
in Assumption~\ref{assum-standing} and (\ref{psi0bound}), such that
$\mathbb{T}^t_h$ satisfies the monotonicity in Theorem~\ref{teo-BS}\textup{(ii)} for all $h\in(0, h_0]$.
\end{lem}

\begin{pf} Let $\varphi_1\le\varphi_2$ be bounded and $\psi:=
\varphi_2 - \varphi_1\ge0$. Then by (\ref{Th}) we have, at $(t,x)$,
%
%
\begin{equation}
\label{DTh1} \qquad \mathbb{T}^t_h[ \varphi_2] -
\mathbb{T}^t_h[ \varphi_1] =
\mathcal{D}^{t,0}_h \psi+ h \bigl[F_y
\mathcal{D}^{t,0}_h \psi+ F_z \cdot
\mathcal{D}^{t,1}_h \psi+ F_\gamma\dvtx
\mathcal{D}^{t,2}_h \psi \bigr].
\end{equation}
Here the terms $F_y, F_z, F_\gamma$ are defined in an obvious way,
and we emphasize that they are deterministic. Plug (\ref{Dh}) into the
above equality, then
%
%
\begin{eqnarray}
\label{DTh2}
&& \mathbb{T}^t_h[\varphi_2] -
\mathbb{T}^t_h[\varphi_1]\nonumber
\\
&&\qquad = \mathbb {E} [
\psi(x+ \sqrt{h} \sigma_0 \xi) \bigl[ 1 + h \bigl[F_y +
F_z \cdot K_1(\xi) + F_\gamma\dvtx
K_2(\xi) \bigr] \bigr]
\\
&&\qquad = \mathbb{E} \biggl[\psi(x+ \sqrt{h} \sigma_0 \xi) \biggl[ 1 + h
F_y + \sqrt{h} F_z \cdot \bigl(\sigma_0^{-T}
\xi \bigr) + { I\over
 1-p} \biggr] \biggr],\nonumber 
\end{eqnarray}
%
where
\begin{eqnarray*}
I&:=& F_\gamma\dvtx \bigl(\sigma_0^{-T}
\bigl[(1-p) \xi\xi^\T - (1-3p) D \bigl[\xi \xi^\T \bigr] -
2p I_d \bigr]\sigma_0^{-1} \bigr)
\\
&=& \bigl[\widetilde G_\gamma- \tfrac{1}{2} I_d
\bigr]\dvtx \bigl[(1-p) \xi\xi^\T - (1-3p) D \bigl[\xi
\xi^\T \bigr] - 2p I_d \bigr]
\\
&=&(1-p) \bigl[\widetilde G_\gamma- \tfrac{1}{2}
I_d \bigr]\dvtx \bigl(\xi \xi^\T \bigr) - (1-3p)
\bigl[D[\widetilde G_\gamma] - \tfrac{1}{2} I_d
\bigr]\dvtx \bigl(\xi \xi ^\T \bigr)
\\
&&{} - 2p\operatorname{tr}(\widetilde G_\gamma) + pd.
\end{eqnarray*}
Denote $\alpha_i:= (\widetilde G_\gamma)_{ii}$. Then it follows from
the definition of $\theta$ that
\begin{eqnarray*}
I&\ge& (1-p) \biggl[{1\over 1+\theta}D[\widetilde G_\gamma ] -
\frac{1}{2} I_d \biggr]\dvtx \bigl(\xi\xi^\T
\bigr) - (1-3p) \biggl[D[\widetilde G_\gamma] - \frac{1}{2}
I_d \biggr]\dvtx \bigl(\xi \xi^\T \bigr)
\\
&&{}- 2p\operatorname{tr}(\widetilde G_\gamma) + pd
\\
&=&p\alpha_p D[\widetilde G_\gamma]\dvtx \bigl(\xi
\xi^\T \bigr) - p|\xi|^2- 2p\operatorname{tr}( \widetilde G_\gamma) + pd
\\
&=& pd - p\sum_{i=1}^d \bigl[
\xi_i^2 + 2\alpha_i - \alpha_p
\alpha_i \xi_i^2 \bigr].
\end{eqnarray*}
Note that $\unde \alpha\le\alpha_i \le\bar\alpha$,
$\alpha_p \ge2p$ by Remark~\ref{rem-mon1}(i), $\unde \alpha
= c_p^{-2}$ by Remark~\ref{rem-mon1}(ii) and $\xi_i^2$ takes only
values $0$ and ${1\over  p}$. Then
\begin{eqnarray*}
p \bigl[ \xi_i^2 + 2\alpha_i -
\alpha_p \alpha_i \xi_i^2 \bigr]
& \le& (2p\alpha_i) \vee \bigl[1 -(\alpha_p-2p)
\alpha_i \bigr]
\\
&\le& {2p\bar\alpha} \vee \bigl[1-(\alpha_p-2p) \unde
\alpha \bigr] = 2p\bar\alpha= 2p \Lambda c_p^{-2}.
\end{eqnarray*}
%
Thus
%
%
\begin{eqnarray}
\label{DTh3} 1-p + I&\ge& 1-p + pd - 2pd \Lambda c_p^{-2}
= {\lambda_p |\sigma
_0^{-1}| \over \sqrt{p}} \ge{\lambda^* |\sigma
_0^{-1}|\over
2\sqrt{p}}.
\end{eqnarray}
By the Lipschitz continuity of $G$ in $\gamma$, we have
%
%
\begin{equation}
\label{C0} C\bigl|\sigma_0^{-1}\bigr|^2 \ge
\bar\alpha= \Lambda c_p^{-2} = {\Lambda\over 2p\Lambda+ \alpha_p - 2p}
\ge {1\over \alpha_p} \ge{1\over 3}.
\end{equation}
%
Note that $p\le{1\over 3}$ and $|\xi_i|$ takes
values $0$ or ${1\over
\sqrt{p}}$. Then
%
%
\begin{equation}
\label{DTh4} \bigl|h F_y + \sqrt{h} F_z \cdot \bigl(
\sigma_0^{-1}\xi \bigr) \bigr| \le Ch + {C\sqrt{h}\over \sqrt{p}}
\bigl|\sigma_0^{-1}\bigr| \le {C_1\sqrt{h_0}\over
\sqrt{p}}\bigl|
\sigma_0^{-1}\bigr|,
\end{equation}
for\vspace*{1pt} some constant $C_1$. Set $h_0:= ({3\lambda^*\over
 4C_1})^2$, and
recall again that $p\le{1\over 3}$ and (\ref
{psi0bound}). Plugging
(\ref{DTh3}) and (\ref{DTh4}) into (\ref{DTh2}), we see that
\begin{eqnarray*}
&& 1 + h F_y + \sqrt{h} F_z \cdot \bigl(
\sigma_0^{-1}\xi \bigr) + { I\over
 1-p} \ge0
\end{eqnarray*}
and thus proves the monotonicity.
\end{pf}

We remark that our algorithm works well when $\widetilde G_\gamma$ is
diagonally dominated, namely when $\theta$ is uniformly small. In
this case, we have the following simple sufficient condition for
Assumption~\ref{assum-mon}.

%
\begin{prop}
\label{prop-mon}
Let Assumption~\ref{assum-standing} hold. Assume there exist $\sigma
_0\dvtx\break  [0, T]\times\mathbb{R}^d \to\mathbb{R}^{d\times d}_+$ and a
small constant $\varepsilon_0>0$ such that $\sigma_0$ and $\sigma
_0^{-1}$ are bounded and, for the $C_0$ in (\ref{C0}),
%
%
\begin{equation}
\label{assum-mon2} \theta\le{\varepsilon_0\over 4dC_0} \quad\mbox{and}\quad
{\unde \alpha(t,x)\over |\sigma_0^{-1}|^2} \ge\varepsilon_0.
\end{equation}
%
Then Assumption~\ref{assum-mon} holds, and consequently $\mathbb
{T}^t_h$ is monotone.
\end{prop}

\begin{pf} First, set $p:= {\varepsilon_0\over
 4dC_0} \in[{\theta
\over 2(1+\theta)}, {1\over
 3}] \cap(0, {1\over
 3}]$. It is clear
that (\ref{psi0bound}) holds. By the first inequality of (\ref{C0}) and
second inequality of (\ref{assum-mon2}), we have $\Lambda\le
{C_0\over \varepsilon_0}$. Then $\alpha_p \ge
{p(1+3\theta)\over
p(1+\theta)} \ge1$, $c_p \ge\sqrt{\alpha_p } \ge1$, and thus,
\[
\lambda_p\ge\sqrt{p} \biggl[1-p+pd - 2pd{C_0\over
 \varepsilon
_0}
\biggr] \varepsilon_0 \ge\sqrt{\varepsilon_0\over
 4dC_0}
\frac{1}{2}\varepsilon_0= {\varepsilon_0^{3/ 2}\over  C}.
\]
This implies Assumption~\ref{assum-mon} immediately.
\end{pf}

%
\begin{rem}
\label{rem-mon2} In this remark we investigate the bound of
$\Lambda$ for our algorithm, and compare it with (\ref{FTWconstraint}).
\begin{longlist}[(iii)]
\item[(i)] When $\widetilde G_\gamma$ is diagonally dominated, in particular
when $\theta= 0$, under (\ref{assum-mon2}) we remove the constraint
(\ref{FTWconstraint}) completely and thus improve the result of \cite
{FTW} significantly. We also note that when $d=1$ we always have
$\theta=0$ and thus the bound constraint does not exist in our case.

\item[(ii)] Let $0<\theta\le2$ and $d\ge2$. For simplicity, we shall
assume $\unde \alpha, \bar\alpha$ and $\theta$ are all
constants. Note that
\begin{eqnarray*}
\lambda_p &=& {\sqrt{p\unde \alpha}\over
c_p|\sigma
_0^{-1}|} \bigl[(1-p+pd) (2p\Lambda+
\alpha_p-2p) - 2pd\Lambda \bigr]
\\
&=& 2p(1-p)\sqrt{p} (d-1) {\sqrt{p\unde \alpha}\over
c_p|\sigma_0^{-1}|} \biggl[{(1-p+pd)(\alpha_p-2p)\over
 2p(1-p)
(d-1)}
- \Lambda \biggr]
\\
&=&2p(1-p) (d-1){\sqrt{p\unde \alpha}\over
c_p|\sigma_0^{-1}|} \biggl[ \biggl[1 +
{1\over (d-1)p} \biggr] \biggl[ 1- {\theta\over
 2(1+\theta)p} \biggr] -
\Lambda \biggr].
\end{eqnarray*}
Then our constraint is
%
%
\begin{equation}
\label{Lth} \Lambda< \Lambda_\theta:= \sup_{p\in[{\theta/ (2(1+\theta
))}, {1/ 3}]}
\biggl[1 + {1\over
 (d-1)p} \biggr] \biggl[ 1- {\theta\over
 2(1+\theta
)p}
\biggr].
\end{equation}
When $0<\theta\le{2\over  d+3}$, one may
compute straightforwardly
that the optimal $p:={2\theta\over 2-(d-3)\theta
} \in[{\theta
\over 2(1+\theta)}, {1\over
 3}]$ and thus
$
\Lambda_\theta= 1+{[2-(d-3)\theta]^2\over
 8\theta(1+\theta)(d-1)}$.
Once again, we see that $\Lambda_\theta$ is large when $\theta$
is small. In particular, there exists unique $\theta_d \le{2\over
d+3}$ such that $\Lambda_{\theta_d} = 1+{2\over
  d}$. Therefore,
when $\theta< \theta_d$, our scheme allows for a larger bound of
$\Lambda$ than (\ref{FTWconstraint}).

\item[(iii)] When\vspace*{1pt} $\theta\ge{2\over  d+3}$, or more
generally when $\theta
\ge\theta_d$, we may set $p:= {1\over 3}$ and
our algorithm reduces back to \cite{FTW}, by replacing the Brownian
motion with trinomial tree; see Remark~\ref{rem-p}.
In this case Assumption~\ref{assum-mon} may be violated, but we can
still easily obtain the same bound (\ref{FTWconstraint0}) as in \cite
{FTW}. That is, under (\ref{FTWconstraint0}) our algorithm (with
$p={1\over 3}$) is still monotone, but the
proof should follow the
arguments in \cite{FTW}, rather than those in Lemma~\ref{lem-mon}.

\item[(iv)] By\vspace*{1pt} (\ref{Lth}), to maintain monotonicity it suffices to choose $p$
such that
$
[1 + {1\over (d-1)p}][ 1- {\theta\over
 2(1+\theta)p}] \ge\Lambda$.
In particular, when $\theta=0$, one natural choice is
\[
p:= \min \biggl({1\over (\Lambda-1)(d-1)}, {1\over 3} \biggr).
\]
We remark further that, in light of Remark~\ref{rem-mon1}(iii), we
may not want to choose smaller $p$ although it also maintains the monotonicity.
\end{longlist}
\end{rem}

%
\begin{rem}
\label{rem-mon3} This remark concerns the degeneracy of $G$.
\begin{longlist}[(iii)]
\item[(i)] The\vspace*{2pt} second inequality of (\ref{assum-mon2}) implies immediately the
uniform nondegeneracy of $G_\gamma$: $D[\sigma_0^{-1} G_\gamma
\sigma_0^{-T}] \ge\varepsilon_0|\sigma_0^{-1}|^2 I_d$. This is
mainly due to the term $F_z \cdot(\sigma_0^{-1} \xi)$ in (\ref
{DTh2}). In \cite{FTW}, $G_\gamma$ is assumed to be nondegenerate,
but not necessarily uniformly, under the additional assumption that
$F_z^\T F_\gamma^{-1} F_z$ is bounded ($F_\gamma$~is nondegenerate
in \cite{FTW}). If we assume that $|F_z|\le C |\unde \alpha|$,
then following similar arguments, in particular by using a weaker
version of monotonicity in the spirit of~\cite{FTW}, Lemma 3.12, we may
remove the uniform nondegeneracy in (\ref{assum-mon2}) as well.

\item[(ii)] Unlike \cite{FTW}, we do not require $G_\gamma\ge\frac{1}{2}
a_0$ and thus $F$ can be degenerate. This is possible mainly because we
use a bounded trinomial tree instead of an
unbounded Brownian motion.

\item[(iii)] When $G$ is degenerate, namely $\unde {\alpha}$ can be
equal to $0$, one can approximate the generator $G$ by
$G^\varepsilon:= G+ \varepsilon I_d\dvtx  \gamma$ and numerically solve
the corresponding solution $u^\varepsilon$. By the stability of
viscosity solutions we see that $u^\varepsilon$ converges to $u$
locally uniformly.

\item[(iv)] Motivated from pricing Asian options, in a recent work Tan \cite
{Tan2} investigated the numerical approximation for the following type
of PDE with solution $u(t,x,y)$:
\[
-\partial_t u(t,x,y)- G \bigl(t,x,y, u, D_x u,
D_{xx}^2u \bigr) - H(t,x,y,u, D_xu,
D_y u) =0,
\]
where $G$ is nondegenerate in $D^2_{xx} u$, but the PDE is always
degenerate in $D^2_{yy} u$.
\end{longlist}
\end{rem}

\subsection{Stability}\label{sec3.3}
Given the monotonicity, one may prove stability following standard arguments.

%
\begin{lem}
\label{lem-stability}
Let Assumptions~\ref{assum-standing} and~\ref{assum-mon} hold. Then
for any $h\in(0, h_0]$, $\mathbb{T}^t_h$ satisfies the stability in
Theorem~\ref{teo-BS}\textup{(iii)}.
\end{lem}

\begin{pf} First, it follows from Lemma~\ref{lem-mon} that $\mathbb
{T}^t_h$ satisfies the monotonicity. Denote\vspace*{1pt} $C_n:= \sup_{x\in\mathbb
{R}^d} |u_h(t_n, x)|$ and $C_i:= \sup_{(t, x)\in[t_i,
t_{i+1})\times\mathbb{R}^d} |u_h(t,\break x)|$, $i=n-1,\ldots, 0$. Since
$g$ is bounded, we see that $C_n \le C$. We claim that
%
%
\begin{equation}
\label{stability-claim} C_i \le(1+ Ch) C_{i+1} + Ch.
\end{equation}
Then by the discrete Gronwall inequality, we see that
\begin{eqnarray*}
&&\max_{0\le i\le n-1} C_i \le C (1+Ch)^n [
C_n +nh] \le C.
\end{eqnarray*}
This proves the lemma.

We now prove (\ref{stability-claim}). Let $(t,x)\in[t_i,
t_{i+1})\times\mathbb{R}^d$ and denote $h':= t_i - t \le h$, $\xi:=
\xi^{t,x}$.
Similar to (\ref{DTh1}), one may easily get
%
%
\begin{eqnarray}
\label{contdef} u_h(t, x) &=& \mathbb{E} \bigl[u_h
\bigl(t_{i+1}, x +\sqrt{ h'}\sigma_0 \xi
\bigr) I_{i+1} \bigr] + h' F (t_i, x, 0, {
\mathbf0}, { \mathbf0} ),
\end{eqnarray}
where, for some deterministic $F_y(t_i), F_z(t_i), F_\gamma(t_i)$
defined in an obvious way,
\begin{eqnarray*}
I_{i+1} &:=& 1+ h' \bigl[F_y(t_i)
+ F_z(t_i)\cdot K_1(\xi) +
F_\gamma (t_i)\dvtx K_2(\xi) \bigr].
\end{eqnarray*}
The monotonicity in Lemma~\ref{lem-mon} exactly means $I_{i+1}\ge0$.
Noting that  $F(t_i, x,\break 0, {\mathbf0}, {\mathbf0} ) = G(t_i, x, 0, {\mathbf0},
{\mathbf0})$ is bounded, then
\begin{eqnarray*}
\bigl|u_h(t, x)\bigr| &\le& \mathbb{E} \bigl[\bigl|u_h
\bigl(t_{i+1}, x +\sqrt{h'}\sigma _0 \xi
\bigr)\bigr| I_{i+1} \bigr] + C h' \le C_{i+1}
\mathbb{E}[ I_{i+1}] + C h'.
\end{eqnarray*}
By (\ref{EK}), we see that
$
E_{t_i}[I_{i+1}] = 1+ h' F_y(t_i) \le1+ C h'$.
Then
\begin{eqnarray*}
\bigl|u_h(t, x)\bigr| &\le& \bigl(1+C h' \bigr)C_{i+1}
+ C h'\le(1+C h)C_{i+1} + C h.
\end{eqnarray*}
Since $(t,x)$ is arbitrary, we obtain (\ref{stability-claim}).
\end{pf}

\subsection{Boundary condition}\label{sec3.4}
%


%
\begin{lem}
\label{lem-boundary}
Let Assumptions~\ref{assum-standing} and~\ref{assum-mon} hold, then
\[
\bigl|u_h(t,x)-g(x)\bigr| \leq C(T-t)^{1/2}. %
\]
\end{lem}

\begin{pf} Without loss of generality, we assume $t=t_k$ for some $k$.
Fix $(t_k, x)$, and denote $X^n_{t_k}:= x$, $\mathcal{F}^n_{t_k}:= \{
\varnothing, \Omega\}$. For $j=k+1,\ldots, n$, define recursively
\[
X^n_{t_{j}}:= X^n_{t_{j-1}} +\sqrt{h}
\sigma_0 \bigl(t_{j-1}, X^n_{t_{j-1}}
\bigr) \xi^{j},\qquad\mathcal{F}^n_{t_{j}}:=
\mathcal {F}^n_{t_{j-1}} \vee\sigma \bigl(\xi^{j}
\bigr),
\]
where $\xi^{j}:= \xi^{t_{j-1}, X^n_{t_{j-1}}}$ is determined by
(\ref{xii}) and is independent of $\mathcal{F}^n_{t_{j-1}}$.
Then it is clear,
\begin{eqnarray*}
u_h \bigl(t_{j-1}, X^n_{t_{j-1}} \bigr)
&=& \mathbb{E}_{t_{j-1}} \bigl[u_h \bigl(t_j,
X^n_{t_{j}} \bigr) \bigr]
\\
&&{} + h F \bigl(t_{j-1},
X^n_{t_{j-1}}, \mathbb {E}_{t_{j-1}}
\bigl[u_h \bigl(t_j, X^n_{t_{j}}
\bigr) \bigr],
\\
&&\hspace*{32pt} \mathbb{E}_{t_{j-1}} \bigl[u_h \bigl(t_j,
X^n_{t_{j}} \bigr) K_1 \bigl(\xi^j
\bigr) \bigr], \mathbb{E}_{t_{j-1}} \bigl[u_h
\bigl(t_j, X^n_{t_{j}} \bigr)K_2
\bigl( \xi^j \bigr) \bigr] \bigr).
\end{eqnarray*}
Similar to the proof of Lemma~\ref{lem-stability}, we have
\begin{eqnarray*}
u_h \bigl(t_{j-1}, X^n_{t_{j-1}} \bigr)
&=& \mathbb{E}_{t_{j-1}} \bigl[u_h \bigl(t_j,
X^n_{t_{j}} \bigr) I_j \bigr] + h F
\bigl(t_{j-1}, X^n_{t_{j-1}}, 0, {\mathbf0}, {
\mathbf0} \bigr),
\end{eqnarray*}
where, by abusing the notation $I$ slightly,
\begin{eqnarray*}
I_j &:=& 1 + h \bigl[ F_y(t_{j-1}) +
F_z(t_{j-1}) \cdot K_1 \bigl(
\xi^j \bigr) + F_\gamma(t_{j-1})\dvtx
K_2 \bigl( \xi^j \bigr) \bigr] \ge0,
\end{eqnarray*}
and $F_y(t_{j-1}), F_z(t_{j-1}), F_\gamma(t_{j-1})$ are defined in an
obvious way. Denote
\[
J_k:= 1\quad\mbox{and}\quad J_i:= \prod_{j=k+1}^{i}
I_j,\qquad i=k+1,\ldots, n.
\]
Recalling $u_h(t_n, x) = g(x)$, by induction we get
\[
u_h(t_k, x) = u_h \bigl(t_k,
X^n_{t_k} \bigr) = \mathbb{E} \Biggl[ g
\bigl(X^n_{t_n} \bigr) J_n + h \sum
_{j=k}^{n-1} J_{j} F \bigl(t_{j},
X^n_{t_{j}}, 0, { \mathbf0}, {\mathbf 0} \bigr) \Biggr].
\]
Since $g$ is bounded and uniformly Lipschitz continuous, we may let
$g_\epsilon$ be a standard smooth molifier of $g$ such that
%
%
\begin{eqnarray}
\label{ge} \qquad && \|g_\epsilon-g\|_\infty\leq C\epsilon,\qquad
\|Dg_\epsilon\|_\infty \leq C\quad\mbox{and}\quad
\bigl\|D^2g_\epsilon\bigr\|_\infty \leq C\epsilon^{-1}.
\end{eqnarray}
Then, noting again that $F(t,x,0,{\mathbf0}, {\mathbf0}) = G(t,x,0,{\mathbf0},
{\mathbf0})$ is bounded, 
\begin{eqnarray*}
&&\bigl|u_h(t_k, x) - g(x)\bigr|
\\
&&\qquad \le \bigl| \mathbb{E} \bigl[ g_\varepsilon \bigl(X^n_{t_n}
\bigr)J_{n} \bigr] - g_\varepsilon(x) \bigr| + C\varepsilon\mathbb{E} [
J_{n} ] + Ch\mathbb{E} \Biggl[ \sum_{j=k}^{n-1}
J_{j} \Biggr]+ C\varepsilon
\\
&&\qquad = \Biggl| \sum_{i=k+1}^n \mathbb{E} \bigl[
g_\varepsilon \bigl(X^n_{t_i} \bigr) J_i -
g_\varepsilon \bigl(X^n_{t_{i-1}} \bigr) J_{i-1}
\bigr] \Biggr| +C\varepsilon\mathbb{E}[ J_{n}] + Ch\mathbb{E} \Biggl[ \sum
_{j=k}^n J_{ j} \Biggr]+ C
\varepsilon
\\
&&\qquad = \Biggl|\sum_{i=k+1}^n \mathbb{E} \bigl[
\bigl[g_\varepsilon \bigl(X^n_{t_i} \bigr) -
g_\varepsilon \bigl(X^n_{t_{i-1}} \bigr) \bigr]
J_i + g_\varepsilon \bigl(X^n_{t_{i-1}}
\bigr)J_{i-1} [I_i-1] \bigr] \Biggr|
\\
&&\quad\qquad{}+ C\varepsilon\mathbb{E} [ J_{n} ] + Ch\mathbb{E} \Biggl[ \sum
_{j=k}^{n-1} J_{j} \Biggr]+ C
\varepsilon.
\end{eqnarray*}
Since $\mathbb{E}_{t_{j-1}}[I_j] = 1+ h F_y(t_{j-1})$, we have
%
%
\begin{equation}
\label{EIj} \bigl|\mathbb{E}_{t_{i-1}}[I_i] - 1 \bigr|\le Ch\quad
\mbox{and}\quad \mathbb {E} [ J_i ] \le(1+Ch)^{i-k} \le C.
\end{equation}
Thus
%
%
\begin{eqnarray}
\label{uh-g} && \bigl|u_h(t_k, x) - g(x)\bigr|
\nonumber
\\[-8pt]
\\[-8pt]
&&\qquad\leq \Biggl|\sum_{i=k+1}^n \mathbb{E}
\bigl[ \bigl[g_\varepsilon \bigl(X^n_{t_i} \bigr) -
g_\varepsilon \bigl(X^n_{t_{i-1}} \bigr) \bigr]
J_i \bigr] \Biggr| + C(n-k)h + C\varepsilon.
\nonumber
\end{eqnarray}
Moreover, for some appropriate $\mathcal{F}_{t_i}$-measurable $\widetilde X^n_{t_i}$,
\begin{eqnarray*}
g_\varepsilon \bigl(X^n_{t_i} \bigr) - g_\varepsilon
\bigl(X^n_{t_{i-1}} \bigr) &=&\sqrt h Dg_\epsilon
\bigl(X^n_{t_{i-1}} \bigr) \cdot\sigma_0
\xi^i +{h\over 2} D^2 g_\epsilon
\bigl(\widetilde X^n_{t_i} \bigr)\dvtx \bigl(
\sigma_0 \xi ^i \bigr) \bigl(\sigma_0
\xi^i \bigr)^\T.
\end{eqnarray*}
By (\ref{psi0bound}), we have
\begin{eqnarray*}
&& \bigl|\mathbb{E}_{t_{i-1}} \bigl[Dg_\epsilon \bigl(X^n_{t_{i-1}}
\bigr) \cdot \sigma_0 \xi^i I_i \bigr] \bigr|
\\
&&\qquad = \bigl|\mathbb{E}_{t_{i-1}} \bigl[h \bigl[Dg_\epsilon
\bigl(X^n_{t_{i-1}} \bigr) \cdot\sigma_0
\xi^i \bigr] \bigl[ F_z(t_{i-1})\cdot
K_1 \bigl(\xi^i \bigr) \bigr] \bigr] \bigr| \leq C\sqrt{h};
\\
&& \bigl|\mathbb{E}_{t_{i-1}} \bigl[D^2 g_\epsilon \bigl(
\widetilde X^n_{t_i} \bigr)\dvtx \bigl(\sigma_0
\xi^i \bigr) \bigl(\sigma_0\xi^i
\bigr)^\T I_i \bigr] \bigr|\le C \varepsilon^{-1}
\mathbb{E}_{t_{i-1}}[ I_i] \le C\varepsilon^{-1}.
\end{eqnarray*}
Then
\[
\bigl|\mathbb{E}_{t_{i-1}} \bigl[ \bigl[g_\varepsilon \bigl(X^n_{t_i}
\bigr) - g_\varepsilon \bigl(X^n_{t_{i-1}} \bigr) \bigr]
I_i \bigr] \bigr| \le Ch \bigl[1+\varepsilon^{-1} \bigr].
\]
Plugging this into (\ref{uh-g}) and recalling (\ref{EIj}), we have
\begin{eqnarray*}
\bigl|u_h(t_k, x) - g(x)\bigr| &\le& C(n-k)h \bigl[1+
\varepsilon^{-1} \bigr] + C\varepsilon.
\end{eqnarray*}
Note that $(n-k)h = T-t_k$. Setting $\varepsilon:= \sqrt{T-t_k}$, we
obtain the result.
\end{pf}


\subsection{Convergence results}\label{sec3.5}

First, combine Lemmas~\ref{lem-consistency},~\ref{lem-mon},~\ref
{lem-stability},~\ref{lem-boundary}, and immediately from Theorem~\ref
{teo-BS}, we have the following:

%
\begin{teo}
\label{teo-conv}
Let Assumptions~\ref{assum-standing} and~\ref{assum-mon} hold. Then
the PDE (\ref{PDE}) has a unique bounded viscosity solution $u$, and
$u_h$ converges to $u$ locally uniformly as $h\to0$.
\end{teo}

We next study the rate of convergence. We first consider the case that
$u$ is smooth. Let $C^{[2]}_b([0,T]\times\mathbb{R}^d)$ denote the
subset of $C^{1,2}([0,T]\times\mathbb{R}^d)$ such that $u$,
$\partial_t u, D u, D^2u$ are bounded; and $C^{[4]}_b([0,T]\times
\mathbb{R}^d)$ the set of $u\in C^{[2]}_b([0,T]\times\mathbb{R}^d)$
such that each component of $\partial_t u, D u, D^2u$ is also in
$C^{[2]}_b([0,T]\times\mathbb{R}^d)$.

\begin{teo}
\label{teo-smooth}
Let Assumptions~\ref{assum-standing} and~\ref{assum-mon} hold and
$h\in(0, h_0)$. Assume further that $u\in C_b^{[4]}([0,T]\times
\mathbb{R}^d)$, and $G$ is locally uniformly Lipschitz continuous in
$x$, locally uniformly on $(y,z,\gamma)$. Then there exists a
constant $C$, independent of $h$ (or $n$), such that
\begin{eqnarray*}
\bigl|u_h(t,x)-u(t,x)\bigr|&\le &Ch\qquad\mbox{for all } (t, x).
\end{eqnarray*}
\end{teo}

\begin{pf} Again, since $h < h_0$, it follows from Lemma~\ref{lem-mon}
that $\mathbb{T}^t_h$ satisfies the monotonicity. Denote $C_n:= \sup_{x\in\mathbb{R}^d} |u_h(t_n, x) - u(t_n,x)|$ and $C_i:= \sup_{(t,
x)\in[t_i, t_{i+1})\times\mathbb{R}^d} |u_h(t,x)- u(t,x)|$,
$i=n-1,\ldots, 0$. We claim that
%
%
\begin{equation}
\label{smooth-claim} C_i \le(1+ Ch) C_{i+1} + Ch^2.
\end{equation}
Since $C_n=0$, then by the discrete Gronwall inequality, we see that
\[
\max_{0\le i\le n-1} C_i \le C (1+Ch)^n \bigl[
C_n +n h^2 \bigr] \le Ch.
\]
This proves the theorem.

We now prove (\ref{smooth-claim}). Similar to the proof of Lemma
\ref
{lem-stability}, we shall only estimate $|u_h(t_i,x)-u(t_i, x)|$, and
the estimate for the general $|u_h(t,x)-u(t, x)|$ is similar. For this
purpose, recall (\ref{Th}), (\ref{Dh}), and define
%
%
\begin{eqnarray}
\label{tildeuh} \tilde{u}_h(t_i,x)&:=& \bigl[
\mathcal{D}^{t_i,0}u(t_{i+1},\cdot) \bigr](x)\nonumber
\\
&&{} +h F \bigl(t_i, x, \bigl[\mathcal{D}^{t_i,0}u(t_{i+1},
\cdot) \bigr](x), \bigl[\mathcal{D}^{t_i,1}u(t_{i+1},\cdot)
\bigr](x),
\\
&&\hspace*{146pt}{}  \bigl[\mathcal{D}^{t_i,2} u(t_{i+1},\cdot) \bigr](x)
\bigr). \nonumber
\end{eqnarray}
We note that the right-hand side of above uses the true solution $u$,
instead of $u_h$ in~(\ref{BS-uh}). It is clear that
%
%
\begin{equation}
\label{Duh} \qquad \bigl|u_h(t_i,x)-u(t_i,x)\bigr|
\le\bigl|u_h(t_i,x)-\tilde{u}_h(t_i,x)\bigr|+\bigl|
\tilde {u}_h(t_i,x)-u(t_i,x)\bigr|.
\end{equation}

Compare (\ref{BS-uh}) and (\ref{tildeuh}), and by the first equality of
(\ref{DTh1}) we have, at $(t_i,x)$,
\begin{eqnarray*}
&&u_h(t_i,x)-\tilde{u}_h(t_i,x)
\\
&&\qquad = \mathbb{E} \bigl[[u_h-u](t_{i+1}, x + \sqrt{h}
\sigma_0 \xi) \bigl[ 1 + h \bigl[F_y + F_z
\cdot K_1(\xi) + F_\gamma\dvtx K_2(\xi) \bigr]
\bigr]\bigr].
\end{eqnarray*}
Then it follows from similar arguments in the proof of Lemma~\ref
{lem-stability} that
%
%
\begin{equation}
\label{Duh1} \bigl|u_h(t_i,x)-\tilde{u}_h(t_i,x)
\bigr| \le(1+Ch) C_{i+1}.
\end{equation}
Next, since $u\in C_b^{[4]}([0,T]\times\mathbb{R}^d)$, applying
Taylor expansion and by (\ref{Exii}), one may check straightforwardly
that, at $(t_i,x)$,
\begin{eqnarray*}
\bigl[\mathcal{D}^{t_i,0} u(t_{i+1},\cdot) \bigr](x) &=& u+
\partial _t u h + {h\over 2} D^2 u\dvtx
a_0 + O \bigl(h^2 \bigr);
\\
\bigl[\mathcal{D}^{t_i,1} u(t_{i+1}, \cdot) \bigr](x) &=& D u +
O(h);\qquad \bigl[\mathcal{D}^{t_i,2} u(t_{i+1}, \cdot) \bigr](x)
= D^2 u+ O(h),
\end{eqnarray*}
where $O(\cdot)$ is uniform, thanks to (\ref{psi0bound}). Then, again
at $(t_i, x)$,
\begin{eqnarray*}
\tilde{u}_h-u &=& \partial_t u h +
{h\over 2} D^2 u\dvtx a_0 + O
\bigl(h^2 \bigr)
\\
&&{} + h F \bigl(t_i, x, u + O(h), D u + O(h), D^2 u+O(h)
\bigr)
\\
&=& \partial_t u h - {h\over 2} O(h)\dvtx
a_0 + O \bigl(h^2 \bigr)
\\
&&{} + h G \bigl(t_i, x, u + O(h), D u + O(h), D^2 u+O(h)
\bigr).
\end{eqnarray*}
Note that $u$ satisfies the PDE (\ref{PDE}), and recall (\ref{F}). Then
\begin{eqnarray*}
&& \bigl|\tilde{u}_h(t_i,x)-u(t_i,x) \bigr|
\\
&&\qquad = O
\bigl(h^2 \bigr) + h
\bigl|G \bigl(\cdot, u + O(h), D u + O(h), D^2 u+O(h) \bigr)
\\
&&\hspace*{181pt} {}- G
\bigl(\cdot, u, D u, D^2 u \bigr) \bigr|(t_i, x).
\end{eqnarray*}
Since $u$ and its derivatives are bounded, and $G$ is locally uniformly
Lipschitz continuous in $x$, then we have
\begin{eqnarray*}
\bigl|\tilde{u}_h(t_i,x)-u(t_i,x) \bigr| &\le& C
h^2.
\end{eqnarray*}
Plug this and (\ref{Duh1}) into (\ref{Duh}), we obtain
\begin{eqnarray*}
\sup_x \bigl|{u}_h(t_i,x)-u(t_i,x)\bigr|
&\le& (1+Ch) C_{i+1} + Ch^2.
\end{eqnarray*}
Similarly we may estimate $\sup_x |{u}_h(t,x)-u(t,x)|$ for $t \in
(t_i, t_{i+1})$ and thus prove~(\ref{smooth-claim}).
\end{pf}


We finally study the case when $u$ is only a viscosity solution. Given
the monotonicity, our arguments are almost identical to those of Fahim,
Touzi and Waxin~\cite{FTW}, Theorem 3.10, which in turn relies on the
works Krylov \cite{Krylov} and Barles and Jakobsen \cite{BJ}. We thus
present only the result and omit the proof.

The result relies on the following additional assumption.

%
\begin{assum}
\label{assum-HJB}
(i) $G$ is of the Hamilton--Jocobi--Bellman type
\begin{eqnarray*}
G(t,x,y,z,\gamma) &=& \inf_{\alpha\in\mathcal{A}} \biggl[\frac
{1}{2}
\sigma^{\alpha} \bigl(\sigma^\alpha \bigr)^\T (t,x)\dvtx
\gamma+b^{\alpha}(t,x)y+c^{\alpha
}(t,x)\cdot z+f^{\alpha}(t,x)
\biggr],
\end{eqnarray*}
where $\sigma^{\alpha}$, $b^{\alpha}$, $c^{\alpha}$ and $f^{\alpha}$
are uniformly bounded, and uniformly Lipschitz continuous in $x$ and
uniformly H\"{o}lder-$\frac{1}{2}$ continuous in
$t$, uniformly in
$\alpha$.

(ii) For any $\alpha\in\mathcal{A}$ and $\delta>0$, there exists a
finite set $\{\alpha_i\}_{i=1}^{M_{\delta}}$ such that
\[
\inf_{1\leq i\leq M_{\delta}} \bigl(\bigl|\sigma^{\alpha}-\sigma
^{\alpha_i}\bigr|_{\infty}+\bigl|b^{\alpha}-b^{\alpha_i}\bigr|_{\infty
}+\bigl|c^{\alpha}-c^{\alpha_i}\bigr|_{\infty}+\bigl|f^{\alpha}-f^{\alpha
_i}\bigr|_{\infty}
\bigr)\leq\delta. %
\]
\end{assum}

We then have the following result analogous to \cite{FTW}, Theorem 3.10.

%
\begin{teo}
\label{teo-viscosity}
Let Assumptions~\ref{assum-standing} and~\ref{assum-mon} hold and
$h\in(0, h_0)$:
\begin{longlist}[(ii)]
\item[(i)] under Assumption~\ref{assum-HJB}\textup{(i)}, we have $u-u_h\leq Ch^{1/4}$,
\item[(ii)] under the full Assumption~\ref{assum-HJB}, we have
$-Ch^{1/10}\leq u-u_h\leq Ch^{1/4}$.
\end{longlist}
\end{teo}

%
\begin{rem}\label{rem-Tan}
The arguments in \cite{FTW} rely heavily on the viscosity
properties of the PDE. Very recently Tan \cite{Tan} provides a purely
probabilistic arguments for HJB equations. His argument works for the
non-Markovian setting as well and thus provides a discretization for
second order BSDEs. Moreover, under his conditions he shows that $|u_h
- u|\le C h^{1/ 8}$, which improves the
left-hand side rate in
Theorem~\ref{teo-viscosity}(ii).
\end{rem}

\section{Quasi-linear PDE and coupled FBSDEs}\label{sec4}
\label{sect-FBSDE}

In this section we focus on the following $G$ which is quasi-linear in
$\gamma$:\footnote{The idea of rewriting this PDE in the form of
(\ref{F}) for numerical purpose was communicated to the second author
by Nizar Touzi back in 2003, which was in fact a main motivation to
study the second order BSDE in \cite{CSTV,STZ}.}
%
%
\begin{eqnarray}
\label{quasi} G & =& \tfrac{1}{2} \bigl[\sigma\sigma^\T
\bigr](t,x,y)\dvtx \gamma+ b \bigl(t,x,y,\sigma(t,x,y)z \bigr) \cdot z
\nonumber\\[-8pt]\\[-8pt]
&&{} + f \bigl(t,x,y,\sigma(t,x,y)z \bigr).
\nonumber
\end{eqnarray}
Here $f$ is scalar, $b$ is $\mathbb{R}^d$-valued, and $\sigma$ is
$\mathbb{R}^{d\times m}$-valued for some $m$.
In this case the PDE (\ref{PDE}) is closely related to the following
coupled FBSDE:
%
%
\begin{equation}
\label{FBSDE} \cases{ \displaystyle X_t= x+\int
_0^t b(s, X_s, Y_s,
Z_s)\,ds+\int_0^t\sigma (s,
X_s,Y_s)\,dW_s;
\cr
\displaystyle
Y_t=g(X_T)+\int_t^\T
f(s,X_s, Y_s, Z_s)\,ds- \int
_t^\T Z_s \cdot \,dW_s.}
\end{equation}
Here $W$ is a $m$-dimensional Brownian motion, and $(X, Y, Z)$ is the
solution triplet taking values in $\mathbb{R}^d$, $\mathbb{R}$, and
$\mathbb{R}^m$, respectively. Due to the four-step scheme of Ma,
Protter and Yong \cite{MPY}, when the PDE (\ref{PDE}) has the
classical solution, the following nonlinear Feynman--Kac formula holds:
%
%
\begin{equation}
\label{FK} Y_t = u(t, X_t),\qquad Z_t =
\sigma \bigl(t, X_t, u(t, X_t) \bigr) D u (t,
X_t).
\end{equation}

The feasible numerical method for high-dimensional FBSDEs has been a
challenging problem in the literature. There are very few papers on the
subject (e.g., \cite{BZ,CZ,DM,MPY,MSZ,Makarov,MT}), most of which are
not feasible in high-dimensional cases. To our best knowledge, the only
work which reported a high-dimensional numerical example is Bender and
Zhang \cite{BZ}.

Our scheme works for quasi-linear PDE as well, especially when $\sigma
\sigma^\T$ is diagonally dominated, and thus is appropriate for
numerically solving FBSDE (\ref{FBSDE}). We remark that the $\sigma
_0(t,x)$ we will choose is different from $\sigma(t,x,y)$ in (\ref
{quasi}), and the $F$ defined by (\ref{F}) is different from $f$. We
shall present a 12-dimensional example; see Example~\ref{eg-FBSDE} below.

One\vspace*{2pt} technical point is that the $G$ in (\ref{quasi}) is not Lipschitz
continuous in $y$, mainly due to the term $\frac{1}{2} [\sigma\sigma
^\T](t,x,y)\dvtx  \gamma$. This can be overcome when the PDE has a
classical solution $u\in C^{[4]}_b([0, T]\times\mathbb{R}^d)$ (as in
Theorem~\ref{teo-smooth}).

%
\begin{teo}
\label{teo-FBSDE}
Let $G$ take the form (\ref{quasi}). Assume:
\begin{longlist}[(iii)]
\item[(i)] $\sigma, b, f, g$ are bounded, continuous in all variables, and
uniformly Lipschitz continuous in $(x,y,z)$.

\item[(ii)] Assumption~\ref{assum-mon} holds.

\item[(iii)] The PDE (\ref{PDE}) has a classical solution $u\in C^{[4]}_b([0,
T]\times\mathbb{R}^d)$.

Then $|u_h - u|\le Ch$ when $h$ is small enough.
\end{longlist}
\end{teo}

\begin{pf} We follow the proof of Theorem~\ref{teo-smooth}. Define
$C_i$, $i=0,\ldots, n$ and $\tilde u_h$ as in Theorem~\ref
{teo-smooth} and again it suffices to prove (\ref{smooth-claim}).

We first estimate $|u_h(t_i, x) - \tilde u_h(t_i,x)|$. Denote
\[
\varphi_h(x):= u_h (t_{i+1},x),\qquad
\varphi(x):= u(t_{i+1},x),\qquad\psi:= \varphi_h - \varphi.
\]
Then, denoting $\mathcal{D}^j:= \mathcal{D}^{t_i, j}$, $j=0,1,2$,
and suppressing the variables $(t_i,x)$,
\begin{eqnarray*}
u_h - \tilde u_h &=& \mathcal{D}^0\psi+h
F \bigl(\cdot, \mathcal {D}^0\varphi_h,
\mathcal{D}^1\varphi_h, \mathcal{D}^2
\varphi _h \bigr) - h F \bigl(\cdot, \mathcal{D}^0
\varphi, \mathcal {D}^1\varphi, \mathcal{D}^2 \varphi
\bigr)
\\
&=& \mathcal{D}^0\psi- {h\over 2} a_0
\dvtx \mathcal {D}^2 \psi+h G \bigl(\cdot, \mathcal{D}^0
\varphi_h, \mathcal{D}^1\varphi_h,
\mathcal{D}^2 \varphi_h \bigr)
\\
&&{} - h G \bigl(\cdot, \mathcal
{D}^0\varphi, \mathcal{D}^1\varphi,
\mathcal{D}^2 \varphi \bigr)
\\
&=& \mathcal{D}^0\psi+ h \bigl[ \mathcal{L}^1\psi+
\mathcal {L}^2\psi+ \mathcal{L}^3\psi \bigr],
\end{eqnarray*}
where, denoting $a(t,x,y):= \sigma\sigma^\T(t,x,y)$,
\begin{eqnarray*}
\mathcal{L}^1\psi&:=&- \tfrac{1}{2} a_0
\dvtx \mathcal {D}^2 \psi+ \tfrac{1}{2} a \bigl(\cdot,
\mathcal{D}^0\varphi_h \bigr)\dvtx \mathcal{D}^2
\psi + b \bigl(\cdot,\mathcal{D}^0\varphi_h,\sigma
\bigl(\cdot,\mathcal {D}^0\varphi_h \bigr)
\mathcal{D}^1\varphi_h \bigr) \cdot\mathcal
{D}^1\psi;
\\
\mathcal{L}^2\psi&:=& \tfrac{1}{2} \bigl[a \bigl(\cdot,
\mathcal {D}^0\varphi_h \bigr) - a \bigl(\cdot,
\mathcal{D}^0\varphi \bigr) \bigr]\dvtx \mathcal{D}^2
\varphi;
\\
\mathcal{L}^3\psi&:=& \bigl[b \bigl(\cdot,\mathcal{D}^0
\varphi _h,\sigma \bigl(\cdot,\mathcal{D}^0
\varphi_h \bigr)\mathcal{D}^1\varphi _h
\bigr) - b \bigl(\cdot,\mathcal{D}^0\varphi,\sigma \bigl(\cdot,
\mathcal{D}^0\varphi \bigr)\mathcal{D}^1\varphi \bigr)
\bigr]\cdot \mathcal{D}^1 \varphi
\\
&&{}+ \bigl[f \bigl(\cdot,\mathcal{D}^0\varphi_h,\sigma
\bigl(\cdot,\mathcal{D}^0\varphi_h \bigr)
\mathcal{D}^1\varphi_h \bigr) - f \bigl(\cdot,
\mathcal{D}^0\varphi,\sigma \bigl(\cdot,\mathcal {D}^0
\varphi \bigr)\mathcal{D}^1\varphi \bigr) \bigr].
\end{eqnarray*}
Let $\eta$ denote a generic function with appropriate dimension which
is uniformly bounded and may vary from line to line. Since $u\in
C^{[4]}_b([0, T]\times\mathbb{R}^d)$, one may easily check that
$\mathcal{D}^0 \varphi, \mathcal{D}^1\varphi, \mathcal
{D}^2\varphi$ are bounded. Then
\begin{eqnarray*}
\mathcal{L}^1\psi&=& \tfrac{1}{2}a \bigl(\cdot, \mathcal
{D}^0\varphi_h \bigr)\dvtx \mathcal{D}^2
\psi- \tfrac{1}{2} a_0\dvtx \mathcal {D}^2
\psi+ \eta \cdot\mathcal{D}^1\psi;
\\
\mathcal{L}^2 \psi&=&\eta \mathcal{D}^0 \psi;
\\
\mathcal{L}^3\psi&=& \eta\mathcal{D}^0\psi+ \eta\cdot
\bigl[\sigma \bigl(\cdot,\mathcal{D}^0\varphi_h \bigr)
\mathcal{D}^1\varphi_h- \sigma \bigl(\cdot,
\mathcal{D}^0\varphi \bigr)\mathcal{D}^1\varphi \bigr]
\\
&=& \eta\mathcal{D}^0\psi+ \eta\cdot\sigma \bigl(\cdot,\mathcal
{D}^0\varphi_h \bigr)\mathcal{D}^1\psi+
\eta\cdot \bigl[\sigma \bigl(\cdot,\mathcal{D}^0\varphi_h
\bigr) -\sigma \bigl(\cdot,\mathcal{D}^0\varphi \bigr) \bigr]
\mathcal{D}^1\varphi
\\
&=& \eta\mathcal{D}^0\psi+ \eta\cdot\mathcal{D}^1\psi.
\end{eqnarray*}
Thus
\begin{eqnarray*}
u_h - \tilde u_h &=& \mathcal{D}^0\psi+ h
\bigl[ \eta\mathcal {D}^0\psi+ \eta\cdot\mathcal{D}^1
\psi \bigr]+ {h\over
 2} [a- a_0 ]\dvtx
\mathcal{D}^2 \psi
\\
&=& \mathbb{E} \biggl[\psi(x+\sqrt{h}\sigma_0\xi) \biggl[ 1+ h\eta+
h \eta\cdot K_1(\xi) + {h\over 2}[a -
a_0]\dvtx K_2(\xi) \biggr] \biggr].
\end{eqnarray*}
Now following the same arguments as in Lemma~\ref{lem-mon}, for small
$h$ we have
\[
1+ h\eta+ h \eta\cdot K_1(\xi) + {h\over 2}[a -
a_0]\dvtx K_2(\xi) \ge0.
\]
Then it follows from the arguments in Theorem~\ref{teo-smooth} that
\[
\bigl|u_h(t_i, x) - \tilde u_h(t_i,x)\bigr|
\le(1+Ch) C_{i+1}.
\]
Similarly we may prove
$
|\tilde u_h(t_i, x) - u(t_i,x)| \le Ch^2$.
Thus we prove (\ref{smooth-claim}) and hence the theorem.
\end{pf}

%
\begin{rem}
\label{rem-FBSDE}
(i) The existence of classical solutions for quasi-linear PDEs can
be seen in \cite{LSU}. The rationale for the convergence in this case
is as follows. Let $C_0$ be a bound for $u$, $Du$, $D^2 u$ and assume
$d=1$ for simplicity. Let $\hat z:= (-C_0)\vee z \wedge C_0$ and $\hat
\gamma:= (-C_0)\vee\gamma\wedge C_0$ be the truncation. Consider
%
%
\begin{eqnarray}
\label{hatG} \widehat G(\cdot,z,\gamma) &:=& \tfrac{1}{2} a\dvtx \hat
\gamma+ b (\cdot,\sigma\hat z ) \cdot\hat z+ f (\cdot,\sigma \hat z ).
\end{eqnarray}
Then $u$ is a classical solution of the following PDE as well:
%
%
\begin{equation}
\label{hatFBSDE} -\partial_t u - \widehat G \bigl(t,x,u, Du,
D^2 u \bigr) =0.
\end{equation}
Under the conditions of Theorem~\ref{teo-FBSDE}, one can easily check
that $\widehat G$ satisfies Assumption~\ref{assum-standing}. However, we
should point out that $\widehat G$ violates the nondegeneracy required in
Assumption~\ref{assum-mon}, so one still cannot apply Theorem~\ref
{teo-smooth} directly on PDE (\ref{hatFBSDE}).

(ii) If the PDE has a classical solution, by applying the so called
partial comparison (comparison between classical semisolution and
viscosity semisolution), which is much easier than the comparison
principle for viscosity solutions, one can easily see that $u$ is
unique in viscosity sense.

(iii) In general viscosity solution cases, even if the PDE is wellposed
in the viscosity sense, we are not able to extend Theorems~\ref
{teo-conv} and~\ref{teo-viscosity} directly, because $G$ violates the
uniform Lipschitz continuity. However, if one can approximate $G$ by
certain $G_\varepsilon$ and the PDE with generator $G_\varepsilon$
has classical solution, then following the stability of viscosity
solutions we may numerically approximate $u$, in the spirit of Remark
\ref{rem-mon3}(iii).
\end{rem}

\section{Implementation of the scheme}\label{sec5}
\label{sect-MC}

In this section we discuss how to implement the scheme. Fix $x_0\in
\mathbb{R}^d$, $0=t_0<\cdots<t_n=T$, and some desirable $\sigma_0$
and~$p$, our goal is to numerically compute $u_h(t_0, x_0)$. Define
$(X^n_{t_i}, \mathcal{F}^n_{t_i})$ as in the proof of Lemma~\ref
{lem-boundary}. That is, denote $X^n_{t_0}:= x_0$, $\mathcal
{F}^n_{t_0}:= \{\varnothing, \Omega\}$, and define recursively: for
$i=0,\ldots, n-1$,
%
%
\begin{equation}
\label{Xn} X^n_{t_{i+1}}:= X^n_{t_{i}}
+\sqrt{h} \sigma_0 \bigl(t_{i}, X^n_{t_{i}}
\bigr) \xi^{i+1},\qquad\mathcal{F}^n_{t_{i+1}}:=
\mathcal{F}^n_{t_i} \vee \sigma \bigl(\xi^{i+1}
\bigr),
\end{equation}
where $\xi^{i+1}:= \xi^{t_i, X^n_{t_i}}$ is determined by (\ref
{xii}) [corresponding to $p(t_i, X^n_{t_i})$] and is independent of
$\mathcal{F}^n_{t_i}$. We next define $Y^n_{t_n}:= g(X^n_{t_n})$, and
for $i=n-1,\ldots, 0$,
%
%
\begin{eqnarray}
\label{Yn} \qquad Y^n_{t_i}&:=& \mathbb{E}_{t_i}
\bigl[Y^n_{t_{i+1}} \bigr]
\nonumber\\[-8pt]\\[-8pt]
&&{} + h F \bigl(t_i, X^n_{t_i},
\mathbb{E}_{t_i} \bigl[Y^n_{t_{i+1}} \bigr],
\mathbb{E}_{t_i} \bigl[Y^n_{t_{i+1}}K_1
\bigl(\xi^{i+1} \bigr) \bigr], \mathbb {E}_{t_i}
\bigl[Y^n_{t_{i+1}} K_2 \bigl(\xi^{i+1}
\bigr) \bigr] \bigr),\nonumber
\end{eqnarray}
where the kernels $K_1$ and $K_2$ are defined in (\ref{Dh}). Then one
can easily check
%
%
\begin{equation}
\label{Ynuh} Y^n_{t_i} = u_h
\bigl(t_i, X^n_{t_i} \bigr).
\end{equation}
In\vspace*{2pt} particular, $u_h(t_0, x_0) = Y^n_{t_0}$, and thus it suffices to
compute $Y^n_{t_0}$. Clearly, the main issue is to compute efficiently
the conditional expectations in (\ref{Yn}).

\subsection{Low-dimensional case}\label{sec5.1}\label{sect-low}

If $p$ and $\sigma_0$ are constants, then the forward process
$(X^n_{t_i})_{0\le i\le n}$ form a trinomial tree with $\sum_{i=0}^n
(2i+1)^d$ nodes. When the dimension is low, say $ d\le3$, we may
generate the whole trinomial tree and compute the exact value of $Y^n$
(and hence $u_h$) at each node, where the conditional expectations in
(\ref{Yn}) are computed by the weighted average. This method is very
efficient and the result is deterministic. It is in fact comparable to
the standard finite difference method. We remark that Bonnans and
Zidani \cite{BZ2} proposed an improved finite difference scheme for
HJB equations. We will implement our algorithm on Example~\ref
{3dFiniteExample} below with dimension 3.

When $p$ and $\sigma_0$ vary for different $(t,x)$, the number of
nodes in the trinomial tree $(X^n_{t_i})_{0\le i\le n}$ grows
exponentially in $n$, and the above exact method is not feasible
anymore. Similarly, the number of nodes will grow exponentially in $d$
and thus this method also becomes infeasible when $d$ is high, even if
$p$ and $\sigma_0$ are constants. In these cases we will use least
square regression combined with Monte Carlo simulation to approximate
the conditional expectations in (\ref{Yn}). This method has been widely
used in the literature; see, for example, Longstaff and Schwartz \cite
{LS} and Gobet, Lemor and Warin \cite{GLW}, and will be the subject of
the next subsections.

\subsection{Least square regression}\label{sec5.2}
For each $i=0,\ldots, n-1$, fix an appropriate set of basis functions
$e^{i}_j\dvtx  \mathbb{R}^d \to\mathbb{R}$, $j=1,\ldots, J_{i}$.
Typically we set $J_{i}$ and $e^{i}_j$ independent of $i$, but in
general they may vary for different $i$. For $k=0,1,2$ and any function
$\varphi\dvtx  \mathbb{R}^d \to\mathbb{R}$, let $\mathcal
{P}^i_{k}(\varphi)$ denote the least regression function of $\varphi
$ on the linear span of $\{e^{i}_j, 1\le j\le J_{i}\}$ as follows:
%
%
\begin{eqnarray}
\label{Proj} \mathcal{P}^i_{k}(\varphi)&:=& \sum
_{j=1}^{J_{i}} \alpha ^{i,k}_{j}
e^{i}_j\qquad\mbox{where }
\nonumber\\[-8pt]\\[-8pt]
\qquad \bigl\{\alpha_{j}^{i,k} \bigr\}_{1\le j\le J_{i}}&:=& \arg\min
_{\{\alpha_j\}_{1\le j\le J_{i}}}\mathbb{E} \Biggl[ \Biggl|\sum_{j=1}^{J_{i}}
\alpha_j e^{i}_j \bigl(X^n_{t_i}
\bigr)-\varphi \bigl(X^n_{t_{i+1}} \bigr) K_k
\bigl( \xi^{i+1} \bigr) \Biggr|^2 \Biggr].\nonumber
\end{eqnarray}
We then define $u^J_h(t_n,\cdot):=g$, and for $i=n-1,\ldots, 0$,
%
%
\begin{eqnarray}
\label{regression} u^J_h(t_i, x)&:=&
\mathcal{P}^i_{0} \bigl(u^J_h(t_{i+1},
\cdot) \bigr)
\nonumber
\\
&&{}+ h F \bigl(t_i, x, \mathcal{P}^i_{0}
\bigl(u^J_h(t_{i+1},\cdot) \bigr),
\mathcal{P}^i_{1} \bigl(u^J_h(t_{i+1},
\cdot) \bigr), \mathcal {P}^i_{2} \bigl(u^J_h(t_{i+1},
\cdot) \bigr) \bigr).
\end{eqnarray}

Assume we have actually chosen a countable set of basis functions\break
$(e^{i}_j)_{j\ge0}$ satisfying
%
%
\begin{eqnarray}
\label{regressioncondition} &&\lim_{J\to\infty} \inf_{\{\alpha_j\}_{1\le j\le J}}
\mathbb {E} \Biggl[ \Biggl|\sum_{j=1}^{J}
\alpha_j e^{i}_j \bigl(X^n_{t_i}
\bigr)-\mathbb {E}_{t_i} \bigl[u_h \bigl(t_{i+1},
X^n_{t_{i+1}} \bigr)K_k \bigl(\xi^{i+1}
\bigr) \bigr] \Biggr|^2 \Biggr]
\nonumber\\[-8pt]\\[-8pt]
&&\qquad =0\qquad\forall(i,k).\nonumber
\end{eqnarray}
Then, following the arguments in \cite{CLP} or \cite{GLW}, one can
easily show that
%
%
\begin{equation}
\label{regressionconv} \lim_{J\to\infty} u^J_h(t_0,x_0)
= u_h(t_0,x_0).
\end{equation}
The rate of convergence in (\ref{regressionconv}) depends on that in
(\ref{regressioncondition}). Since the focus of this paper is the
monotone scheme (in terms of the time discretization), we omit the
detailed analysis of the convergence (\ref{regressionconv}).

Clearly, it is crucial to find good basis functions. Notice that the
conditional expectations in (\ref{regressioncondition}) are
approximations of $u(t_i, \cdot)$, $D u(t_i, \cdot)$, and $D^2
u(t_i,\cdot)$, respectively. Ideally, in the case that the true
solution $u$ is smooth, we want to choose $(e^{i}_j)_{1\le j\le J_{i}}$
whose linear span include $u(t_i, \cdot)$, $D u(t_i, \cdot)$ and
$D^2 u(t_i,\cdot)$. This is of course not feasible in practice since
$u$ is unknown. Another naive choice is to use the indicator functions
of the hypercubes from uniform space discretization. Theoretically this
will ensure the convergence very well. However, in this case the number
of hypercubes will grow exponentially in dimension $d$, and thus the
curse of dimensionality remains exactly as in standard finite
difference method.

In the literature, people typically use orthogonal basis functions such
as Hermite polynomials, which is convenient for solving the optimal
arguments in (\ref{Proj}). There are some efforts to improve the basis
functions; see, for example, the martingale basis functions in Bender
and Steiner \cite{BS2}, and the local basis functions in Bouchard and
Warin \cite{BouWarin}. However, overall speaking to find good basis
functions is still an open problem and is certainly our interest in
future research.

%

\subsection{Monte Carlo simulation}\label{sec5.3}
As standard in the literature of BSDE numerics, in high-dimensional
cases we use Monte Carlo approach to approximate the minimum arguments
$(\alpha^{i,k}_j)_{1\le j\le J_i}$ in (\ref{regression}). To be
precise, we simulate $L$-paths for the forward diffusion $X^n$ and the
corresponding trinomial random variables $\xi^i$, denoted as
$(X^{n,l})_{1\le l\le L}$ and $(\xi^{i,l})_{1\le l\le L}$,
$i=1,\ldots, n$. Then, in the backward induction the expectation in
(\ref{Proj}) is replaced by the sample average
%
%
\begin{equation}
\label{MCregression} {1\over  L} \sum_{l=1}^L
\Biggl[ \Biggl|\sum_{j=1}^{J_{i}}
\alpha_j e^{i}_j \bigl(X^{n,l}_{t_i}
\bigr)-\varphi \bigl(X^{n,l}_{t_{i+1}} \bigr) K_k
\bigl( \xi ^{i+1,l} \bigr) \Biggr|^2 \Biggr].
\end{equation}
This can be easily solved by linear algebra.
Let $u^{J,L}_h$ be defined by (\ref{regression}), but replacing
$\alpha
_{j}^{i,k}$ with $\bar\alpha_{j}^{i,k}$, the optimal arguments of
(\ref{MCregression}). We remark that $u^{J,L}_h$ is random, and by the
law of large numbers,
%
%
\begin{equation}
\label{MCconv} \lim_{L\to\infty} u^{J,L}_h(t_0,x_0)
= u^{J}_h(t_0,x_0),\qquad
\mbox{a.s.}
\end{equation}
Moreover, one may obtain the rate of convergence in the spirit of the
central limit theorem. We refer to \cite{GLW} for more details.

To understand convergence (\ref{MCconv}), it is important to understand
the variance of $u^{J,L}_h$ for given $L$. One can easily see that the
variance in each step of our scheme is $O( d/L)$, which leads to an
$O(n d/L)$ variance in total. As the theoretical rate of convergence
for PDE with smooth solution is $1/n$ (see Theorem~\ref{teo-smooth}),
the standard deviation of the numerical result vanishes in the same
rate only when $L=O(n^3)$. On the other hand, Glasserman and Yu \cite
{Glasserman} illustrated that the number of paths should be of
$O(\exp(J))$, $J$ being the number of basis functions. Hence around $O(
n^3\exp(J) )$ paths are supposed to be sampled in theory. This is
prohibitive, if not impossible in practice, especially when $J$ is
large. However, various examples in next section show that it's
generally feasible to obtain a desirable rate of convergence with much
fewer paths in practice.

Finally we remark that the Monte Carlo method is much less sensitive to
dimensions. For example, it can be seen in next section that we can use
the Monte Carlo method to approximate a 12-dimensional PDE with 160
time steps and 13,333,333 paths, while for finite difference method with
$d=12$, even for 2 time steps the number of grid points already exceeds
13,333,333.

\subsection{Some further comments}\label{sec5.4}
\label{sect-further}

We note that there are three types of errors involved in this algorithm:
\[
\mbox{Total Err${}={}$Discretization Err${}+{}$Regression Err${}+{}$Simulation Err.}
\]

The main contribution of this paper is the introduction of
the new monotone scheme, and thus we have focused our discussion on the
discretization error in Sections~\ref{sect-Scheme} and~\ref
{sect-FBSDE}. We remark that the analysis of the Regression Error and
the Simulation Error is independent of the Discretization Error. Since
this is not the main focus of the present paper, we shall apply
standard procedure for the regression and the simulation steps. In
particular, since we know the true solution for many examples below, we
will include the true solution (and its derivatives) in the basis
functions, thus the numerical results in these examples will reflect
the discretization error and the simulation error only.

We have also tested examples where the basis functions do not include
the true solution (Example~\ref{4dExample}) or the true solution is
unknown (Example~\ref{eg-Isaacs} and the last part of Example~\ref
{12dimensionalExample}). We shall emphasize though, when the true
solution is unknown, the numerical result is an approximation of
$u^J_h$, which roughly speaking is the least square regression of the
true solution $u$ in the span of the basis functions. For fixed basis
functions, the increase of $n$ and $L$ cannot eliminate the Regression
Error and thus the convergence we observe in numerical results does not
necessarily reflect a small total error. Again, a thorough analysis of
the Regression Error, especially a good mechanism for choosing basis
functions, is an important open problem.

Moreover, although our theoretical results hold true only under the
monotonicity Assumption~\ref{assum-mon}, we nevertheless implement our
scheme to some examples which violate Assumption~\ref{assum-mon}, and
thus our scheme may not be monotone. It is interesting to observe the
convergence in these examples as well. As far as we know, a rigorous
analysis of nonmonotone schemes is completely open.

\section{Numerical examples}\label{sec6}
\label{sect-num}

In this section we apply our scheme to various examples.\footnote{All
numerical examples below are computed using a personal Laptop, which is
a core i5 2.50~GHz processor with 8~GB memory.} For simplicity, except
in Example~\ref{4dExample}, we shall choose constant $\sigma_0$ and~$p$, and assume the $\unde \alpha$ and $\bar\alpha$ in
(\ref{th}) are also constants [or more precisely, use $\inf_{(t,x)}
\unde \alpha(t,x)$ and $\sup_{(t,x)} \bar\alpha(t,x)$
instead]. Quite often, we will use the following functions:
%
%
\begin{equation}
\label{SIN} \SIN(t,x):=\sin \Biggl(t+\sum_{i=1}^n
x_i \Biggr),\qquad \COS(t,x):=\cos \Biggl(t+\sum
_{i=1}^n x_i \Biggr).
\end{equation}

\subsection{Examples under monotonicity condition}\label{sec6.1}
In this subsection we consider examples with diagonal $G_\gamma$, and
we shall always choose $\sigma_0$ diagonal, again except Example~\ref
{4dExample}, so $\theta= 0$ and thus there is no constraint on
$\Lambda$; see Remark~\ref{rem-mon2}(i).

We start with a $3$-dimensional example for which we can compute its
values over the trinomial tree by using the weighted averages. We
remark that in this example only the discretization error is involved.

%
\begin{eg}[(A 3-dimensional fully nonlinear PDE)] \label{3dFiniteExample}
%
%
\begin{eqnarray}
\label{EXsigm0Not0}\qquad -\partial_t u- \frac{1}{2}\sup
_{\unde {\sigma}\leq\sigma\leq\bar
{\sigma}} \bigl[ \bigl(\sigma^2 I_d \bigr)
\dvtx D^2u \bigr] +f(t,x,u,Du)&=&0\qquad\mbox{in } [0, T)\times
\mathbb{R}^d,
\nonumber
\\[-8pt]
\\[-8pt]
u(T,x)&=& \SIN(T,x)\qquad\mbox{on } \mathbb{R}^d,
\nonumber
\end{eqnarray}
where $0<\unde {\sigma}<\bar{\sigma}$ are both in $\mathbb
{R}$, and
%
%
\begin{equation}
\label{eg1-f} f(t,x,y,z)=\frac{1}{d}\sum_{i=1}^dz_i
-\frac{d}{2}\inf_{\unde {\sigma}\leq\sigma\leq\bar
{\sigma}} \bigl(\sigma^2y
\bigr).
\end{equation}
\end{eg}

We remark that we set $f$ in this way so that (\ref{EXsigm0Not0}) has
the classical solution: $u = \SIN$,
with which we can verify the convergence of our numerical approximation.


To test its convergence under different nonlinearities, we assume that
$d=3$, $\unde {\sigma}=1$, $\bar{\sigma}=\sqrt{2}$, $\sqrt
{4}$, or $\sqrt{6}$. 
Supposing that $T=0.5$ and $x_0=(5,6,7)$, we know the true solution is
$u(0,x_0)=\sin(5+6+7)\approx-0.750987$.

According to our scheme, when $\widetilde{G}_\gamma$ is diagonal, $\theta
=0$, which implies $\alpha_p=2$ and, recalling Remark~\ref{rem-mon2}(iv), we can choose the following parameters:
\begin{eqnarray*}
\Lambda&=&\frac{\bar{\sigma}{}^2}{\unde {\sigma}{}^2},\qquad p=\min \biggl(\frac{1}{2(\Lambda-1)},
\frac{1}{3} \biggr),
\\
\unde {\alpha}&=&\frac{1}{2p\Lambda+\alpha_p-2p}, \qquad
\sigma _0= \frac{\unde {\sigma}}{\sqrt{2\unde {\alpha}}} I_d.
\end{eqnarray*}
We remark that $\Lambda= 2, 4, 6$, respectively, which violates the
constraint (\ref{FTWconstraint}) and thus the algorithm in \cite{FTW}
may not be monotone.
Denote the number of time partitions by $n$. By applying the weighted
average method we can obtain the results in Figure~\ref{d3sigma16},
where the cost in time increases from 0.1 second to 800 seconds
exponentially as n increases from 20 to 160 linearly. The table in
Figure~\ref{d3sigma16} contains the numerical solutions when
$\bar{\sigma}{}^2=2$ exclusively, while the graph depicts the
errors under three different choices of $\bar{\sigma}{}^2$.

%
\begin{figure}[b]
\begin{tabular}{@{}c@{\qquad}c@{}}

\includegraphics{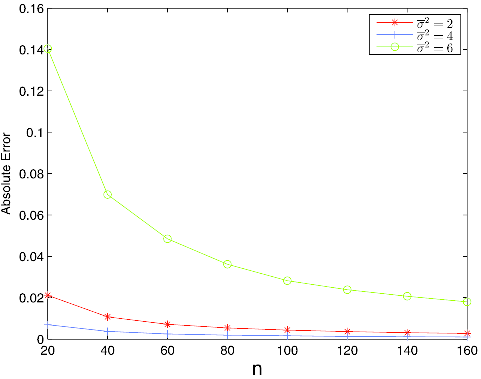}
\\[-134pt]
&
\footnotesize{\begin{tabular}{@{}ld{2.6}@{}}
\hline
 & \multicolumn{1}{c@{}}{\textbf{Approx.}} \\
$\bolds{n}$ & \multicolumn{1}{c@{}}{$\bolds{\bar{\sigma}{}^2=2}$}\\
\hline
\phantom{0}20 & -0.72984 \\
\phantom{0}40 & -0.74028 \\
\phantom{0}60 & -0.74382 \\
\phantom{0}80 & -0.74667 \\
100 & -0.74560\\
120 & -0.74738 \\
140 & -0.74790 \\
160 & -0.74829 \\
Ans. & -0.750987 \\
\hline
\end{tabular}}%
\end{tabular}\vspace*{9pt}
\caption{A 3-dimensional example with various degrees of nonlinearity
in Example \protect\ref{EXsigm0Not0}.}\label{d3sigma16}
\end{figure}

As we can see from Figure~\ref{d3sigma16}, the rate of convergence
is approximately
\mbox{$C\cdot h$}, whereas the $C$ depends on the structure of $G$. Therefore,
our scheme works
generally for large $\Lambda$ when $\widetilde{G}$ is diagonal or
diagonally dominant
with a small $\theta$.

In Figure~\ref{d3ComparisonWithFD} we compare the convergence of our scheme
with that of finite difference method by fixing $\unde {\sigma
}=1$, $\bar{\sigma}=\sqrt{2}$.
It can be seen that our result converges slightly slower than, but is
comparable to, the finite
difference method in solving low-dimensional problems. 

\begin{figure}[t]
\begin{tabular}{@{}c@{\hspace*{14pt}}c@{}}

\includegraphics{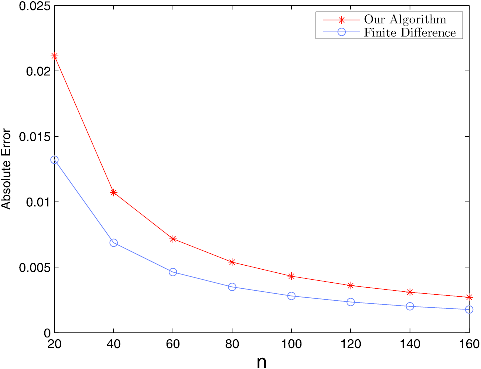}
\\[-122pt]
&
\footnotesize{\begin{tabular}{@{}lcc@{}}
\hline
$\bolds{n}$ & \textbf{Ours} & \textbf{F.D.}\\
\hline
\phantom{0}20 & $-$0.72984 &$-$0.76420\\
\phantom{0}40 & $-$0.74028 &$-$0.75785\\
\phantom{0}60 & $-$0.74382 &$-$0.75562\\
\phantom{0}80 & $-$0.74667 & $-$0.75447\\
100 & $-$0.74560 & $-$0.75379\\
120 & $-$0.74738 & $-$0.75332\\
140 & $-$0.74790 & $-$0.75300\\
160 & $-$0.74829 & $-$0.75274\\
Ans. & $-$0.75099& $-$0.75099 \\
\hline
\end{tabular}}%
\end{tabular}\vspace*{9pt}
\caption{Comparison with finite difference method in Example \protect
\ref{EXsigm0Not0}.}\label{d3ComparisonWithFD}
\end{figure}

To see more of our scheme in extreme condition, we assume $\unde
{\sigma}=0$. Then we truncate $G_\gamma$ from below with a positive
definite matrix $\varepsilon I_d>0$. That is, we approximate (\ref
{EXsigm0Not0}) by the following nondegenerate PDE:
\[
\label{3dExampleTruncated} 
- \partial_t u-\frac{1}{2}\sup
_{{\varepsilon}\leq\sigma\leq
\bar{\sigma}} \bigl[ \bigl(\sigma^2 I_d \bigr)
\dvtx D^2u \bigr] + f(t,x,u,Du)=0,\qquad\varepsilon=0.01,
\]
where $f$ is given by (\ref{eg1-f}) (with $\unde {\sigma}=0$).

%
\begin{figure}[b]
\begin{tabular}{@{}c@{\qquad}c@{}}

\includegraphics{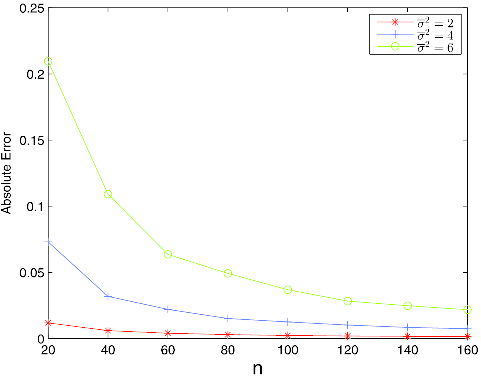}
\\[-134pt]
&
\footnotesize{\begin{tabular}{@{}ld{2.6}@{}}
\hline
 & \multicolumn{1}{c@{}}{\textbf{Approx.}} \\
$\bolds{n}$ & \multicolumn{1}{c@{}}{$\bolds{\bar{\sigma}{}^2=2}$}\\
\hline
\phantom{0}20 & -0.76285 \\
\phantom{0}40 & -0.75705 \\
\phantom{0}60 & -0.75508 \\
\phantom{0}80 & -0.75401 \\
100 & -0.75339\\
120 & -0.75297 \\
140 & -0.75269 \\
160 & -0.75247 \\
Ans. & -0.750987 \\
\hline
\end{tabular}}%
\end{tabular}\vspace*{9pt}
\caption{Convergence of a degenerate PDE truncated in Example \protect
\ref{EXsigm0Not0}.}\label{truncation}
\end{figure}

Figure~\ref{truncation} shows the feasibility of truncation in dealing
with $\unde {\sigma}=0$.


%
\begin{eg}[{[A 4-dimensional PDE with $(\sigma_0, p)$ depending on $(t,x)$]}]\label{4dExample}
%
%
\begin{eqnarray}
\label{4dEquation} \cases{ -\partial_t u -G \bigl(D^2u
\bigr) + f(t,x)=0, &\quad in $[0, T)\times\mathbb{R}^4$,
\vspace*{3pt}\cr
u(T,x) =
\SIN(T,x), &\quad on $\mathbb{R}^4$,}
\end{eqnarray}
where $\SIN
_2:= 2 \SIN \times \COS$, and
\begin{eqnarray*}
\unde {\sigma}&=&\pmatrix{
1 & 0 & 0& 0
\cr
\SIN & 1 & 0 &0
\cr
\COS & \SIN & 1 & 0
\cr
\SIN & \COS & \SIN & 1},\qquad A=\pmatrix{
1 & \displaystyle{2\SIN\over 27} & \displaystyle{\SIN_2 \over 54} & \displaystyle{\SIN_2\over 54}
\vspace*{5pt}\cr
\displaystyle{2\SIN\over 27}& 1 & \displaystyle{2\SIN\over 27} & \displaystyle{\SIN_2 \over 54}
\vspace*{5pt}\cr
\displaystyle{\SIN_2 \over 54} & \displaystyle{2\SIN\over 27}& 1 & \displaystyle{2\SIN\over 27}
\vspace*{5pt}\cr
\displaystyle{\SIN_2 \over 54} & \displaystyle{\SIN_2 \over 54} & \displaystyle{2\SIN\over 27} & 1};
\\[5pt]
\Gamma&=&{4\over 5}\cdot{(6-|\SIN|)(3-2|\SIN|)
\over (3-|\SIN|)^2} =
{4\over 5} \biggl(2-{2|\SIN|\over
 (3-|\SIN|)^2} \biggr)\in \biggl[1,
{8\over
5} \biggr];
\\
\unde  a&:=& \unde {\sigma} \unde {\sigma} {}^\T,\qquad
\bar a = \Gamma \bigl[\unde {\sigma} A \unde {\sigma }
{}^\T \bigr],\qquad G(M):= \max \{ \unde  a\dvtx M, \bar a
\dvtx M \}; 
\end{eqnarray*}
and $f$ is chosen so that $u:= \SIN$ is a classical solution of the PDE.
\end{eg}

We first specify the parameters so that monotonicity Assumption~\ref
{assum-mon} holds. Set $\sigma_0: = \beta\unde {\sigma}$ for
some scalar function $\beta>0$. Then, roughly speaking, $\widetilde G_\gamma$ is either ${1\over \beta^2}I_d$ or ${
\Gamma\over \beta
^2} A$. This implies
%
$D[\widetilde{G}_\gamma] \leq(1+\theta)\widetilde{G}_\gamma$ for
\[
\theta:={2|\SIN| \over 9-2|\SIN| }\leq{2\over
 7}=
{2\over  d+3}.
\]
%
Next, notice that $\Lambda:=\bar{\alpha}/\unde {\alpha
}=\Gamma={4\over 5}\cdot{(6-|\SIN|)(3-2|\SIN|)
\over (3-|\SIN|)^2}$ and
recall~\ref{rem-mon2}(ii). Set
$p:={2-\theta\over 6(1+\theta)}\in[{\theta\over
 2(1+\theta)},
{1\over 3}] \cap(0, {1\over
 3}]$. One can check that
\[
\biggl[1+{1\over (d-1)p} \biggr] \biggl[1-{\theta
\over 2(1+\theta)p}
\biggr] > \Lambda.
\]
We remark that here we do not use $p:={2\theta\over
 2-(d-3)\theta}$
as specified in Remark~\ref{rem-mon2}(ii) because it becomes zero
when $\theta=0$.
Finally, $\beta$ is determined by
%
\begin{eqnarray*}
\beta^2 &=& {1\over \unde \alpha} = c_p^2
= 2p\Lambda+\alpha _p-2p
\\
&=& {1944 -24|\SIN|^2 -1260 |\SIN| \over 270
(3-|\SIN|) } \in \biggl[
{11\over
9}, {12\over 5} \biggr].
\end{eqnarray*}
In particular, we emphasize that here $\sigma_0$ and $p$ depend on
$(t,x)$. 

As explained in Section~\ref{sect-low}, in this case we cannot use the
weighted averages as in previous example. We thus use the combination
of least square regression and Monte Carlo simulation. To illustrate
the important role of the basis functions, we implement our scheme
using three different set of basis functions:
\begin{itemize}
\item the true solution and its derivatives;

\item second order polynomials consisting of\vspace*{1pt}
$
 \{1,\{x_i\}_{i=1}^d$, $\{x_i x_j\}_{1\leq i\leq j\leq d}  \}$;

\item the local basis functions proposed by Bouchard and Warin~\cite{BouWarin}.
\end{itemize}

The idea of local basis functions is as follows. Divide the
samples at each time step into $3^d$ local hypercubes, such that there
are 3 partitions in each dimension, and there are approximately the
same amount of particles in each hypercubes. Then we project samples in
each hypercubes into a linear polynomial of $d+1$ degrees of freedom,
so there are $3^d\cdot(1+d)$ local basis functions in total. Since
each linear polynomial has local hypercube support, the corresponding
matrix in the regression is sparse, making it easier to solve than a
regression problem of dense matrix.

Set $T=0.1$, $x_0= (2,3,4,5)$, and thus the true solution is $\sin
(2+3+4+5)\approx0.9906$. As our \emph{first example} using Monte Carlo
regression, we will sample
$L=3125 n^2$ to see how the convergence works. Moreover, we shall
repeat the tests identically and independently for $K$ times. The
numerical results are reported in Figure~\ref{Figd44dexamplefig},
where the average of the results is denoted as Ans. and the average
time (in seconds) is denoted as Cost.

%
\begin{figure}
{\footnotesize\begin{tabular}{@{}ld{7.0}d{2.0}d{1.4}cccd{1.4}c@{}}
\hline
\multicolumn{3}{@{}l}{\textbf{Basis functions}} & \multicolumn{2}{c}{\textbf{Polynomials}} &
\multicolumn{2}{c}{\textbf{Local basis}} & \multicolumn{2}{c@{}}{\textbf{True solution basis}}
\\[-6pt]
\multicolumn{3}{@{}l}{\hrulefill} & \multicolumn{2}{c}{\hrulefill} &
\multicolumn{2}{c}{\hrulefill} & \multicolumn{2}{c@{}}{\hrulefill}
\\
$\bolds{n}$ & \multicolumn{1}{c}{$\bolds{L}$} & \multicolumn{1}{c}{$\bolds{K}$} & \multicolumn{1}{c}{\textbf{Ans.}}
            & \multicolumn{1}{c}{\textbf{Cost}} & \multicolumn{1}{c}{\textbf{Ans.}} & \multicolumn{1}{c}{\textbf{Cost}}
            & \multicolumn{1}{c}{\textbf{Ans.}} & \multicolumn{1}{c@{}}{\textbf{Cost}}
\\
 \hline
\phantom{0}3 & 28{,}125& 27 & 1.094 &  0.27 & 1.1057 & 0.47 & 1.0959 &  0.17 \\
\phantom{0}5 & 78{,}125 & 16 & 1.0488&  1.6  & 1.0679 & 2.8 & 1.0421 & 0.89\\
 10 & 312{,}500 & 8 & 1.0271 &  16  & 1.0390 & 34 & 1.0123 &8.6\\
 15 & 703{,}125 & 6 & 1.0261 &  58   & 1.0311 & 142 & 1.0008& 30\\
 20 & 1{,}250{,}000 & 4 & 1.0221 &  137  & 1.0258 & 355 & 1.001 & 71\\
 25 & 1{,}953{,}125 & 4 & 1.0240 & 276& 1.0247 & 710 & 0.9986 & 142 \\
 30 & 2{,}812{,}500 & 3 & 1.0228 & 444 & 1.0209 & 1250 & 0.9966 & 243\\
 40 & 5{,}000{,}000 & 2 & 1.0218 &  897   & 1.0156 & 2725 & 0.9952 & 567
\\[3pt]
\multicolumn{3}{@{}l}{True solution} & \multicolumn{2}{l}{0.9906} & \multicolumn{2}{l}{0.9906} & \multicolumn{2}{l@{}}{0.9906}\\
\hline
\end{tabular}
}\vspace*{12pt}%

\includegraphics{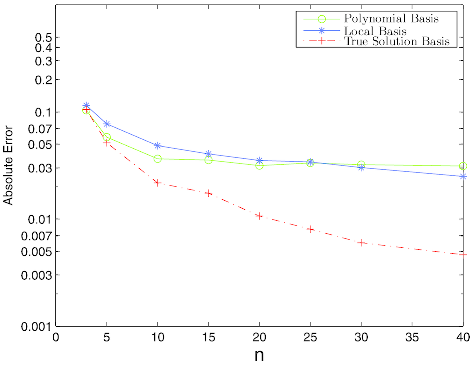}

\caption{Numerical results using different basis functions in Example \protect\ref{4dExample}.}\label{Figd44dexamplefig}
\end{figure}
Without surprise, the true solution basis functions perform the best.
We remark that the results for the other two sets of basis functions
include the regression error as well.
From the numerical results, the local basis functions seem to have
smaller regression error than the polynomials, when $n$ is large.
However, when applying the local basis functions it is time consuming
to sort the $L$ sample paths and localize them into different
hypercubes. When the same number of paths are sampled, the more basis
functions we used, the slower simulation will be. More seriously, when
the dimension $d$ increases, the number of basis functions increases
dramatically, which requires an exponential increase in the number of
paths in return; see \cite{Glasserman} for a detailed investigation of
the relation between basis functions and paths. So further efforts are
needed for higher-dimensional problems.

Our main motivation is to provide an efficient algorithm for
high-dimen\-sional PDEs. At below we test our scheme on a 12-dimensional example, for which we shall again use the
regression-based Monte Carlo method.

%
\begin{eg}[(A 12-dimensional example)]\label{12dimensionalExample}
Consider the PDE (\ref{EXsigm0Not0}) with \mbox{$d=12$},
$\unde {\sigma}=1$, $\bar{\sigma}=\sqrt{2}$,
%
%
\begin{equation}
f(t,x,y,z)=\COS -\frac{d}{2}\inf
_{\unde {\sigma}\leq\sigma\leq
\bar{\sigma}} \bigl(\sigma^2\SIN \bigr).
\end{equation}
%
\end{eg}

The true solution is again $u=\SIN$.
As explained in Section~\ref{sect-further}, in this paper we want to
focus on the discretization error and simulation error, so we rule out
the regression error and test our algorithm by using the following
perfect set of basis functions:
\[
1,\qquad x,\qquad \SIN(T, x),\qquad \COS(T, x).
\]
To test the result, we fix $T=0.2$ and $x_0=\{1,2,\ldots,12\}$, which
implies that the true solution is $\sin(78)=0.513978$. As the
nonlinear term is diagonal, under the same framework as in Example~\ref
{3dFiniteExample}, we take $p:=\min\{1/3, 1/(1+d(\Lambda-1))\}
=1/13$, $\sigma_0:=I_d$, which also satisfy the monotonicity condition
Assumption~\ref{assum-mon}. Assuming that we repeat $K$ identical and
independent tests, and we sample $L$ paths in each test. We do not use
$L=O(N^3)$ in this example and the ones following, since it's usually
not necessary in practice. The results are reported in Figure~\ref
{Figd12sigma12}, where we conduct fewer tests for larger $L$, because
the results are stable enough to draw our conclusion.
%
%
\begin{figure}
{\footnotesize\begin{tabular}{@{}ld{8.0}d{3.0}ccd{1.9}@{}}
\hline
$\bolds{n}$ & \multicolumn{1}{c}{$\bolds{L}$} & \multicolumn{1}{c}{$\bolds{K}$} & \textbf{Avg(Ans.)} & \textbf{Var(Avg.)} & \multicolumn{1}{c@{}}{\textbf{Cost (in seconds)}} \\
\hline
\phantom{00}2 & 2083 & 160 & 0.659639 & $3.53\times10^{-6}$ & 4.48\times10^{-2}\\
\phantom{00}5 & 13{,}021& 64 & 0.562635 & $1.99\times10^{-6}$ & 1.46\times10^{-1}\\
\phantom{0}10 & 52{,}083 & 32 & 0.546598 & $8.41\times10^{-7} $& 1.17\times10^0\\
\phantom{0}20 & 208{,}333 & 16 & 0.530432 & $8.04\times10^{-7}$ & 1.08\times10^1\\
\phantom{0}40 & 833{,}333 & 8 & 0.521343 & $2.25\times10^{-7}$ & 9.11\times10^1 \\
\phantom{0}80 & 3{,}333{,}333 & 4 & 0.519701 & $1.21\times10^{-7}$& 7.28\times10^2 \\
160 & 13{,}333{,}333 & 2 & 0.517363 & $6.17\times10^{-8}$ & 5.86\times 10^3
\\[3pt]
\multicolumn{3}{@{}l}{True solution} & \multicolumn{1}{c}{0.513978}
\\
\hline
\end{tabular}}
\vspace*{12pt}

\includegraphics{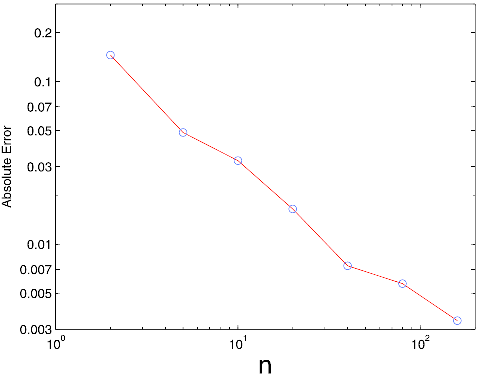}

\caption{Numerical results of a 12-dimensional example in Example \protect\ref{12dimensionalExample}.}\label{Figd12sigma12}
\end{figure}

It can be seen from Figure~\ref{Figd12sigma12} that the error
shrinks slightly slower than $O(h)$, which is due to the simulation
error. Hence we want to explore the influence of simulation error by
using all the parameters as above but fixing $n=40$, $K=2$, $d=12$,
$T=0.2$, $n=40$, $\unde {\sigma}=1$, $\bar{\sigma}=\sqrt{2}$.
We increase the sample size $L$ to see how the error reduces in
Figure~\ref{d12N40increasingL}. %
While the variance and error decrease with more paths sampled, the cost
in time increases linearly with respect to $L$ from 8 seconds to 1400
seconds in Figure~\ref{d12N40increasingL}.

%
%
\begin{figure}[t]
\begin{tabular}{@{}c@{\qquad}c@{}}

\includegraphics{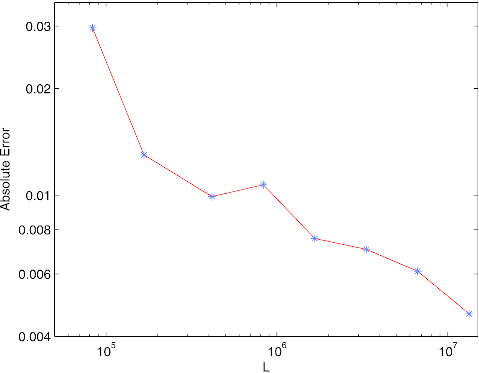}
\\[-125pt]
&
\footnotesize{\begin{tabular}{@{}lc@{}}
\hline
$\bolds{L}$ & \textbf{Approx.} \\
\hline
\phantom{000,}83,333 & 0.543643 \\
\phantom{00,}166,667 & 0.526979 \\
\phantom{00,}416,667 & 0.523897 \\
\phantom{00,}833,333 & 0.524683 \\
\phantom{0}1,666,667 & 0.521531 \\
\phantom{00,}333,333 & 0.521017 \\
\phantom{0}6,666,667 & 0.520083 \\
13,333,333 & 0.518607 \\[3pt]
Ans.: & 0.513978 \\
\hline
\end{tabular}}%
\end{tabular}\vspace*{9pt}
\caption{Relation between size of sample and errors in Example \protect\ref{12dimensionalExample}.}\label{d12N40increasingL}
\end{figure}

We have seen that our scheme converges to the true classical solution
if it exists. Meanwhile, if the PDE only has a unique viscosity
solution, our scheme can render a converging result as well.

Let $f$ be zero in (\ref{EXsigm0Not0}). Then this equation has some
unknown viscosity solution. However, our numerical results in
Figure~\ref{12dViscosity12} still demonstrate a converging sequence.
The number of paths we sampled in~\ref{12dViscosity12} is the same as
that in Figure~\ref{Figd12sigma12}. This can be also be observed
from the decreasing differences between the numerical results. The
$\Delta_{i-j}$ in Figure~\ref{12dViscosity12} denotes a numerical
result with $i$ partitions in time minus another numerical result with
$j$ time steps.
We shall remark though in this case our choice of basis functions may
not be the best, and roughly speaking the numerical result we obtain is
an approximation of the regression of the true solution in the linear
span of the basis functions. Again, we leave the analysis of the basis
functions to future study.
%
%
\begin{figure}[b]
\begin{tabular}{@{}c@{\qquad}c@{}}

\includegraphics{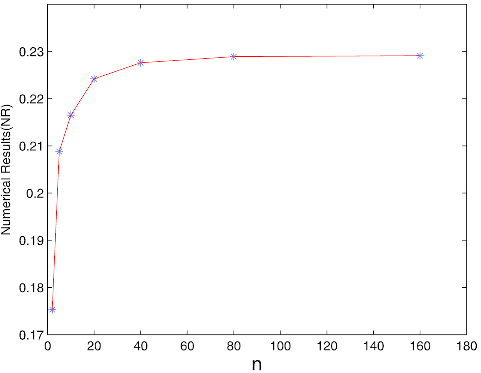}
\\[-74.5pt]
&
\footnotesize{\begin{tabular}{@{}lc@{}}
\hline
$\Delta_{5-2}$ & 0.0334337 \\
$\Delta_{10-5}$ & 0.0077685 \\
$\Delta_{20-10}$ & 0.0076637 \\
$\Delta_{40-20}$ & 0.0034146 \\
$\Delta_{80-40}$ & 0.0012785 \\
$\Delta_{160-80}$ & 0.0002586 \\
\hline
\end{tabular}}%
\end{tabular}\vspace*{9pt}
\caption{Numerical results for a PDE with unknown viscosity solution in Example \protect\ref{12dimensionalExample}.}\label{12dViscosity12}
\end{figure}

It is well known that Isaacs equations have a unique viscosity solution
under mild technical conditions. We next test our scheme on the
following Isaacs equation to see its performance.

%
\begin{eg}[(A 12-dimensional Isaacs equation with unknown viscosity solution)]
\label{eg-Isaacs}
\begin{eqnarray*}
-u_t-G \bigl(D^2u \bigr)&=&0\qquad\mbox{on }[0,T) \times\mathbb{R}^d,
\\
u(T,x) &=& \SIN(T,x)\qquad\mbox{on } \mathbb{R}^d, 
\end{eqnarray*}
where
\begin{eqnarray*}
G(\gamma)&:=&\sum_{i=1}^d\sup
_{0\leq u\leq1}\inf_{{0\leq v\leq1}} \biggl[\frac{1}{2}
\sigma^2(u,v)\gamma _{ii}+f(u,v) \biggr]
\\
&=&\sum_{i=1}^d\inf_{{0\leq v\leq1}}
\sup_{0\leq u\leq1} \biggl[\frac{1}{2}\sigma^2(u,v)
\gamma _{ii}+f(u,v) \biggr],
\\
\sigma^2(u,v)&:=&(1+u+v),\qquad f(u,v):=-\frac{u^2}{4}+
\frac{v^2}{4}.
\end{eqnarray*}
\end{eg}

One can easily simplify $G(\gamma)$ as: $G(\gamma)=\sum_{i=1}^d
g(\gamma_{ii}) $ where
\begin{eqnarray*}
g(\gamma_{ii}) &:=& \cases{ \displaystyle\gamma_{ii}-
\frac{1}{4}, &\quad$\gamma_{ii}\in(1,+\infty)$,
\vspace*{5pt}\cr
\displaystyle
\frac{\gamma_{ii}}{2}+\frac{(\gamma_{ii}^+)^2}{4}-\frac
{(\gamma
_{ii}^-)^2}{4}, &\quad$
\gamma_{ii}\in[-1,1]$,
\vspace*{5pt}\cr
\displaystyle\gamma_{ii}+
\frac{1}{4}, &\quad$\gamma_{ii}\in(-\infty,-1)$.}
\end{eqnarray*}
Therefore $G(\gamma)$ is neither concave nor convex when $\gamma=0$.
Setting $T=0.2$, $d=12$, we assign arbitrary initial value $x_0=\{
x_0^{(i)}\}_{i=1}^{d}$ to inspect the outcome. Obviously here
$\bar{a}=2I_d$ and $\unde {a}=I_d$. We then take $p=\min\{
1/3, 1/(1+d(\Lambda-1))\}=1/13$, $\sigma_0=I_d$.
One example tested here is $x_0^{(i)}=2i\pi-\frac{T-0.5\pi}{d}$. The
number of paths we sampled is $625\cdot n^2$.

Though the viscosity solution is unknown, our scheme still renders a
converging numerical result in Figure~\ref{12dReadViscosity}.
%
%
\begin{figure}
\begin{tabular}{@{}c@{\qquad}c@{}}

\includegraphics{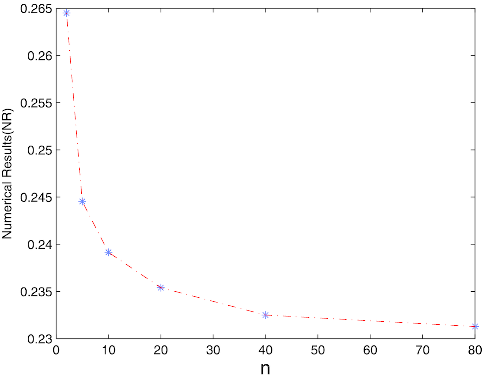}
\\[-74.5pt]
&
\footnotesize{\begin{tabular}{@{}lc@{}}
\hline
$\Delta_{5-2}$ & $-$0.019953 \\
$\Delta_{10-5}$ & $-$0.005390 \\
$\Delta_{20-10}$ & $-$0.003752 \\
$\Delta_{40-20}$ & $-$0.002893 \\
$\Delta_{80-40}$ & $-$0.001213 \\
$\Delta_{160-80}$ & $-$0.000426 \\
\hline
\end{tabular}}%
\end{tabular}\vspace*{9pt}
\caption{A 12-dimensional Isaacs equation with unknown viscosity solution in Example \protect\ref{eg-Isaacs}.}\label{12dReadViscosity}
\end{figure}

We next test our scheme for a $12$-dimensional coupled FBSDE.

%
\begin{eg}[(A 12-dimensional coupled FBSDE)]\label{eg-FBSDE}
Consider FBSDE (\ref{FBSDE}) with $m=d =12$, $\sigma$ is
diagonal, and
\begin{eqnarray*}
b_i(t, x, y, z)&:=&\cos(y+z_i),\qquad
\sigma_{ii}(t,x,y):=1+\frac
{1}{3}\sin \Biggl(\frac{1}{d}
\sum_{j=1}^d x_j+y \Biggr),
\\
f(t,x,y,z)&:=& \frac{d}{2}\SIN(t,x) \Biggl[1+\frac{1}{3}\sin \Biggl(
\frac{1}{d}\sum_{i=1}^d
x_i+y \Biggr) \Biggr]^2
\\
&&{} -\frac{{(1/  d)}\sum_{i=1}^dz_i}{1+({1}/{3})\sin(({1}/{d})\sum_{j=1}^d x_j+y)}
\\
&&{} -d\COS(t,x) \cos \bigl(y+\COS(t,x) \bigr);
\\
g(x)&:=&\SIN(T,x).
\end{eqnarray*}
\end{eg}

The associated PDE (\ref{quasi}) looks quite complicated; however, the
coefficients are constructed in a way so that $u:= \SIN$ is the
classical solution. Consequently, the FBSDE has the following solution:
denoting $\widebar X_t:= {1\over  d}\sum_{j=1}^d X^j_t$,
\[
Y_t=\sin(t+d\widebar X_t ),\qquad Z^i_t=
\cos(t+d\widebar X_t ) \bigl[1+\tfrac
{1}{3}\sin \bigl(\widebar X_t + \sin(t+d\widebar X_t ) \bigr) \bigr].
\]

For PDE (\ref{quasi}), we see that $G_\gamma= \frac{1}{2} \sigma^2$
is diagonal and ${2\over 3} \le\sigma_{ii} \le
{4\over 3}$ for each
$i$. Hence a reasonable choice of parameters that maintains the
monotonicity would be $\sigma_0:={4\over
 9}I_d$, $p:=\min\{ {1 \over
1+ d(\Lambda-1)}$, ${1\over 3} \} = 1/37$.
We\vspace*{1pt} note that $f$ is not bounded and not Lipschitz continuous in $y$;
however, since $Z$ is bounded, then $f(t,x,y, Z_t)$ is bounded and
Lipschitz continuous in $y$, and thus actually we may still apply
Theorem~\ref{teo-FBSDE}. Set $d=12$, $T=0.2$, $X_0=(2,3,4,\ldots,13)$,
and apply the parameters specified before for our scheme. An\vspace*{2pt}
approximation of $Y_0$ is shown in Figure~\ref{d12CoupledFBSDE}, where
the true solution $Y_t=\sin(t+\sum_{i=1}^d X_t^i)$ has value 0.893997
at $t=0$.
%
%
\begin{figure}
{\footnotesize\begin{tabular}{@{}ld{8.0}d{3.0}ccd{1.9}@{}}
\hline
$\bolds{n}$ & \multicolumn{1}{c}{$\bolds{L}$} & \multicolumn{1}{c}{$\bolds{K}$} & \textbf{Avg(Ans.)} & \textbf{Var(Avg.)} & \multicolumn{1}{c@{}}{\textbf{Cost (in seconds)}} \\
\hline
\phantom{00}2 & 2083 & 160 & 1.462543 & $3.35\times10^{-5}$ & 1.56\times10^{-2}\\
\phantom{00}5 & 13{,}021& 64 & 1.111675 & $2.30\times10^{-5}$ & 2.36\times10^{-1}\\
\phantom{0}10 & 52{,}083 & 32 & 1.014295 & $2.48\times10^{-5} $& 2.43\times10^0\\
\phantom{0}20 & 20{,}8333 & 16 & 0.925712 & $8.10\times10^{-6}$ & 2.29\times10^1\\
\phantom{0}40 & 83{,}3333 & 8 & 0.912373 & $2.46\times10^{-6}$ & 1.94\times10^2 \\
\phantom{0}80 & 3{,}333{,}333 & 4 & 0.908013 & $2.89\times10^{-7}$& 1.56\times10^3 \\
160 & 13{,}333{,}333 & 2 & 0.888747 & $1.62\times10^{-8}$ & 3.42\times10^4
\\
\hline
\end{tabular}}
\vspace*{0pt}

\includegraphics{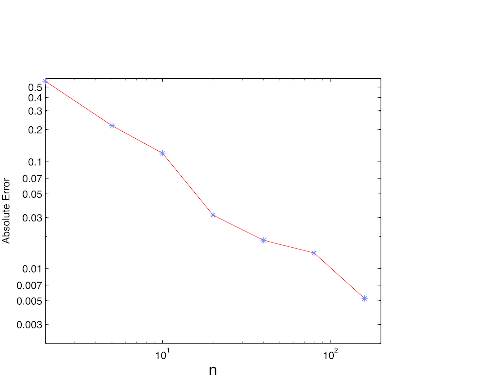}

\caption{A 12-dimensional coupled FBSDE in Example \protect\ref{eg-FBSDE}.}\label{d12CoupledFBSDE}
\end{figure}

\subsection{Examples violating the monotonicity condition}\label{sec6.2}

In this subsection we apply our scheme to some examples which do not
satisfy our monotonicity Assumption~\ref{assum-mon}. So theoretically
we do not know if our scheme converges or not. However, our numerical
results show that the approximation still converges to the true
solution. It will be very interesting to understand the scheme under
these situations, and we shall leave it for future research.

%
\begin{eg}[(A 12-dimensional PDE with $\unde {\sigma}=0$)]\label{eg-degenerate}
Consider the same setting as Example~\ref
{12dimensionalExample} except that $\unde {\sigma}=0$.
\end{eg}

Instead of truncating $G_\gamma$ as we did at the end of Example~\ref
{EXsigm0Not0}, we will pick parameters $p$ and $\sigma_0$ as
if $\unde {\sigma}$ were some small positive number:
$p:=1/d$, $\sigma_0:=\sqrt{2/(2-p)}\bar\sigma I_d$. Then
Assumption~\ref{assum-mon} is violated, and our scheme is in fact not monotone.
Nevertheless, our numerical results show that our approximations still
converge to the true solution if $L:=625 n^2$ paths are used; see
Figure~\ref{d12low2IS0}.

%
%
\begin{figure}[t]
\begin{tabular}{@{}c@{\qquad}c@{}}

\includegraphics{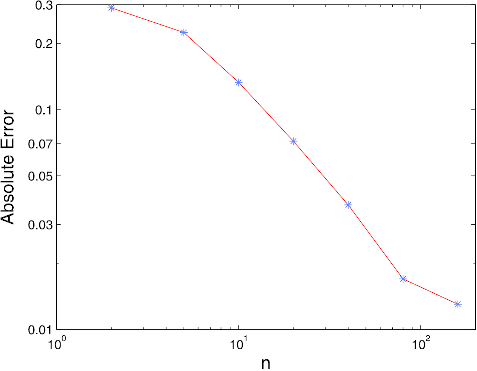}
\\[-111pt]
&
\footnotesize{\begin{tabular}{@{}lc@{}}
\hline
$\bolds{n}$ & \textbf{Approx.} \\
\hline
\phantom{00}2&0.22363\\
\phantom{00}5& 0.28971\\
\phantom{0}10& 0.38098\\
\phantom{0}20& 0.44215\\
\phantom{0}40& 0.47712\\
\phantom{0}80& 0.49699\\
160& 0.50097\\
Ans. & 0.51398 \\
\hline
\end{tabular}}%
\end{tabular}\vspace*{9pt}
\caption{A 12-dimensional example without monotonicity in Example \protect\ref{eg-degenerate}.}\label{d12low2IS0}
\end{figure}
We next apply our scheme to the following HJB equation which is
associated with a Markovian second order BSDEs, introduced by \cite{CSTV,STZ}:
%
%
\begin{equation}
\label{Gexample} \qquad \cases{ \displaystyle-\frac{\partial u}{\partial t}-\frac{1}{2}\sup
_{\unde {\sigma}\leq\sigma\leq\bar{\sigma}} \bigl[\sigma^2\dvtx D^2u \bigr]-
f(t,x)=0,&\quad on $[0,T)\times\mathbb{R}^d$,
\vspace*{5pt}\cr
u(T,x)=g(x),&\quad
on $\mathbb{R}^d$.}
\end{equation}
When $f=0$, this PDE induces exactly the $G$-expectation introduced by
Peng \cite{Peng}. We emphasize that,
unlike in previous examples, here $\unde \sigma, \bar
\sigma, \sigma\in\mathbb{S}^d$ are matrices and ${\mathbf0} <
\unde \sigma\le\sigma\le\bar\sigma$. In particular,
$G_\gamma$ is not diagonal anymore. We remark that one has a
representation for the solution of this PDE in terms of stochastic control,
\begin{eqnarray*}
u(0,x) &=& \sup_{\sigma} \mathbb{E} \biggl[g \bigl(X^\sigma_T
\bigr) + \int_0^\T f \bigl(t,
X^\sigma_t \bigr)\,dt \biggr],\qquad X^\sigma_t:=
x + \int_0^t \sigma _s
\,dW_s,
\end{eqnarray*}
where $W$ is a $d$-dimensional Brownian motion, and the control
$\sigma$ is an $\mathbb{F}^W$-prog\-ressi\-vely measurable $\mathbb
{S}^d$-valued process such that $\unde {\sigma}\leq\sigma\leq
\bar{\sigma}$. Due to this connection, these kind of PDEs and
the related $G$-expectation and second order BSDEs are important in
applications with diffusion control and/or volatility uncertainty.

%
\begin{eg}[(A 10-dimensional HJB equation)]\label{eg-HJB}
Consider the PDE (\ref{Gexample}) with $g(x) =\sin
(T+x_1+{x_2\over 2}+\cdots +{x_{d}\over
  d})$ and appropriate $f(t,x)$ so that
\begin{eqnarray*}
u(t,x)&=&\sin \biggl(t+x_1+{x_2\over 2}+\cdots +
{x_{d}\over  d} \biggr)
\end{eqnarray*}
is the true solution to the PDE. We set $d=10$.
\end{eg}

To begin our test, we select randomly an initial point $X_0$ and two
10-dimen\-sio\-nal positive definite matrices $\bar{\sigma}{}^2$ and
$\unde {\sigma}{}^2$. The parameters used in this PDE are chosen
randomly as:
\begin{eqnarray*}
X_0&=&(0.8870626082, 1.8313582937, 2.1710945122, 2.3703744353,
\\[-3pt]
&&\hspace*{3pt}
1.2018847713, 2.6518851292, 2.2648022663, 1.9037585152,
\\[-3pt]
&&\hspace*{117pt}
2.336892572579084, 1.1590768112),
\end{eqnarray*}
which gives a true solution $-$0.99966,
{\fontsize{9.2}{11}\selectfont{\[
\bar{\sigma} {}^2=
\\
\pmatrix{
2.29 &  0.07 &  -0.48  &  0.15 &  0.89 &  -0.06 &  0.14 &  0.31 &  0.59 &  -0.36
\vspace*{3pt}\cr
 0.07 &  1.82 &  0.55 &  0.32 &  0.28 &  0.08 &  -0.30 &  -0.07 &  -0.46  &  0.66
\vspace*{3pt}\cr
 -0.48 &  0.55 &  2.54 &  -0.35 &  0.14 &  -0.25 &  -0.31 &  0.16 &  -0.71 &  -0.10
\vspace*{3pt}\cr
 0.15 &  0.32 &  -0.35 &  1.71 &  -0.16 &  0.67 &  0.20 &  1.11 &  -0.03 &  -0.64
\vspace*{3pt}\cr
 0.89 &  0.28 &  0.14 &  -0.16 &  1.36 &  -0.47 &  -0.46 &  0.07 &  -0.07 &  -0.03
\vspace*{3pt}\cr
 -0.06 &  0.08 &  -0.25 &  0.67 &  -0.47 &  2.60 &  0.26 &  0.34 &  -0.02 &  -0.67
\vspace*{3pt}\cr
 0.14 &  -0.30 &  -0.31 &  0.20 &  -0.46 &  0.26 &  2.61 &  -0.26 &  0.32 &  0.29
\vspace*{3pt}\cr
 0.31 &  -0.07 &  0.16 &  1.11 &  0.07 &  0.34 &  -0.26 &  2.66 &  -0.19 &  -1.78
\vspace*{3pt}\cr
 0.59 &  -0.46 &  -0.71 &  -0.03 &  -0.07 &  -0.02 &  0.32 &  -0.19 &  1.80 &  -0.43
\vspace*{3pt}\cr
 -0.36 &  0.66 &  -0.10 &  -0.64 &  -0.03 &  -0.67 &  0.29 &  -1.78 &  -0.43 &  2.16}
\]}}%
and
{\fontsize{9.2}{11}\selectfont{\[
\unde {\sigma} {}^2= \pmatrix{
1.53 &  -0.40 &  -0.30 &  -0.20 &  0.66 &  -0.43 &  0.38 &  0.10 &  0.84  &  -0.31
\vspace*{3pt}\cr
 -0.40 &  0.72 &  0.58 &  -0.06 &  -0.15 &  -0.28 &  -0.07 &  -0.14 &  -0.32 &  0.42
\vspace*{3pt}\cr
 -0.30 &  0.58 &  1.55 &  -0.05 &  -0.07 &  -0.54 &  -0.03 &  -0.20 &  -0.51 &  0.38
\vspace*{3pt}\cr
 -0.20 &  -0.06 &  -0.05 &  0.55 &  -0.14 &  0.22 &  -0.09 &  0.60 &  -0.13 &  -0.37
\vspace*{3pt}\cr
 0.66 &  -0.15 &  -0.07 &  -0.14 &  0.61 &  -0.50 &  -0.09 &  -0.10 &  0.25 &  -0.12
\vspace*{3pt}\cr
 -0.43 &  -0.28 &  -0.54 &  0.22 &  -0.50 &  1.27 &  0.15 &  0.34 &  -0.06 &  -0.21
\vspace*{3pt}\cr
 0.38 &  -0.07 &  -0.03 &  -0.09 &  -0.09 &  0.15 &  1.78 &  0.13 &  -0.24 &  0.17
\vspace*{3pt}\cr
 0.10 &  -0.14 &  -0.20 &  0.60 &  -0.10 &  0.34 &  0.13 &  1.04 &  0.16 &  -0.94
\vspace*{3pt}\cr
 0.84 &  -0.32 &  -0.51 &  -0.13 &  0.25 &  -0.06 &  -0.24 &  0.16 &  1.22 &  -0.56
\vspace*{3pt}\cr
 -0.31 &  0.42 &  0.38 &  -0.37 &  -0.12 &  -0.21 &  0.17 &  -0.94 &  -0.56 &  1.36}.
\]}}%

One can check that $\bar{\sigma}{}^2>\unde {\sigma}{}^2$
because the smallest eigenvalue of $\bar{\sigma}{}^2-\unde
{\sigma}{}^2$ is $0.001634$, which is positive.
This PDE is not diagonally dominant, and typically we cannot find
$\sigma_0$ and $p$ to make our scheme monotone. However, it is very
interesting to observe that our scheme converges to the true solution
if we choose $p:=1/3$ and $\sigma_0:={d\sqrt{d}\over
 2\sqrt
{d+1}}\sqrt{\bar{\sigma}{}^2 }$; see Figure~\ref{10dHJB}. We
emphasize again that these parameters still do not satisfy Assumption
\ref{assum-mon}. It will be very interesting to understand further
these numerical results, and we will leave them for future research.
%
%
\begin{figure}
{\footnotesize\begin{tabular}{@{}ld{5.0}d{2.0}d{2.5}cc@{}}
\hline
$\bolds{n}$ & \multicolumn{1}{c}{$\bolds{L}$} & \multicolumn{1}{c}{$\bolds{K}$} & \multicolumn{1}{c}{\textbf{Avg(Ans.)}} & \textbf{Var(Avg.)} & \multicolumn{1}{c@{}}{\textbf{Cost (in seconds)}} \\
\hline
\phantom{0}2 & 283 & 40 & -1.1703 & $9.62\times10^{-7}$ & 0.057 \\
\phantom{0}5 & 1118 & 16 & -1.12773 & $3.80\times10^{-6}$ & 1.9 \\
10 & 3162& 8 &-1.0802 & $5.98\times10^{-6}$ & $23.6$ \\
15 & 5809 & 5 & -1.0557 & $2.32\times10^{-6} $& $103$\\
20 & 8944 & 4 &-1.0405 & $1.57\times10^{-6}$ & $291$ \\
30 & 16{,}432 & 3 & -1.0253 & $9.05\times10^{-6}$& $1135$\\
40 & 25{,}298 & 2 & -1.0124 & $2.16\times10^{-5}$ &$3074$
\\[3pt]
\multicolumn{3}{@{}l}{True solution} & -0.99966
\\
\hline
\end{tabular}}
\vspace*{12pt}

\includegraphics{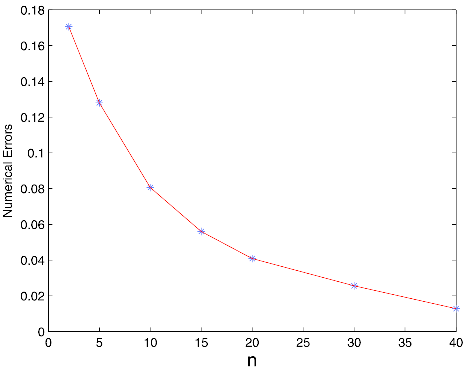}

\caption{A 10-dimensional HJB equation in Example \protect\ref{eg-HJB}.}\label{10dHJB}
\end{figure}

Note that PDE (\ref{Gexample}) involves the computation of $\sup_{\unde {\sigma}\leq\sigma\leq\bar{\sigma}}[\sigma^2\dvtx
\gamma]$. We provide some discussion below.

%
\begin{rem}\label{G-value} Let $\bar{\sigma}{}^2-\unde
{\sigma}{}^2 = LL^\T$ be the Cholesky Decomposition, namely $L$ is a
$d\times d$ lower triangular matrix. Then for any $\gamma\in\mathbb
{S}^d$, we have
\[
\sup_{\unde {\sigma}{}^2\leq\sigma^2\leq\bar{\sigma
}{}^2} \bigl[\sigma^2\dvtx \gamma \bigr]
=\unde {\sigma} {}^2\dvtx \gamma+\sum
_{i=1}^n\hat\gamma_i^+,
\]
where $\hat\gamma_i$, $i=1,\ldots, d$, are the eigenvalues of $L^\T
\gamma L$.
\end{rem}

\begin{pf} Obviously, any $\sigma^2\in\mathbb{S}^d$ between
$\unde {\sigma}{}^2$ and $\bar{\sigma}{}^2$ can be
expressed as $\sigma^2=\unde {\sigma}{}^2+A$, where ${\mathbf0}\leq
A\leq LL^\T$. Then ${\mathbf0}\leq L^{-1} A L^{-T}\leq I_d$. We make the
following eigenvalue decompositions:
\[
L^{-1} A L^{-T} = U\widehat A U^\T,\qquad
L^\T \gamma L = P \hat\gamma P^\T,
\]
where $U U^\T = P P^\T = I_d$, and $\widehat A$ and $\hat g$ diagonal
matrices. It is clear that the diagonal terms of $\widehat A$ are $\hat a_i
\in[0, 1]$, and the diagonal terms of $\hat\gamma$ are $\hat
\gamma_i$. Denote $Q:= U^\T P$. Then
\begin{eqnarray*}
\sigma^2\dvtx \gamma- \unde {\sigma} {}^2\dvtx
\gamma&=& A\dvtx \gamma= \bigl[L^{-1} A L^{-T} \bigr]\dvtx
\bigl[L^\T \gamma L \bigr] = \bigl[U\widehat AU^\T \bigr]\dvtx
\bigl[L^\T \gamma L \bigr]
\\
&=& \widehat A\dvtx \bigl[U^\T L^\T \gamma L U \bigr] = \widehat A
\dvtx \bigl[Q \hat\gamma Q^\T \bigr]
\\
&=& \sum_{i=1}^d \hat a_i
\sum_{j=1}^d q_{ij}^2
\hat\gamma_j \le \sum_{i=1}^d
\Biggl(\sum_{j=1}^d q_{ij}^2
\hat\gamma_j \Biggr)^+.
\end{eqnarray*}
Note that $\sum_{j=1}^d q_{ij}^2 = 1$. Then by Jensen's inequality,
\begin{eqnarray*}
&&\sigma^2\dvtx \gamma- \unde {\sigma} {}^2\dvtx
\gamma\le\sum_{i=1}^d \Biggl(\sum
_{j=1}^d q_{ij}^2 \hat
\gamma_j \Biggr)^+ \le\sum_{i=1}^d
\sum_{j=1}^d q_{ij}^2
\hat\gamma_j^+ = \sum_{j=1}^d
\hat\gamma _j^+ \sum_{i=1}^d
q_{ij}^2 = \sum_{j=1}^d
\hat\gamma_j^+.
\end{eqnarray*}
This proves the remark.

Moreover, from the proof we see that the equality holds when
\[
\hat a_i = 1_{\{\sum_{j=1}^d q_{ij}^2 \hat\gamma_j > 0\}}\quad\mbox{and}\quad Q = I_d.
\]
That is, $U = P$ and thus $\sigma^2 = \unde \sigma^2 + L P
\widehat A P^\T L^\T$, where $\widehat A$ is the diagonal matrix whose diagonal
terms are $\hat a_i = 1_{\{\hat\gamma_i > 0\}}$.
\end{pf}

We remark that the above computation is in fact quite time consuming.
Below we provide another example where $G_\gamma$ is tridiagonal, and
the scheme becomes much more efficient.

%
\begin{eg}[(A 10-dimensional example with tridiagonal structure)]\label{eg-tri}
Consider the PDE (\ref{PDE}) with
%
%
\begin{eqnarray}
\label{FullyNonlinearLastExample} G(t,x,y,z,\gamma) &:=& \Biggl(3\sum
_{i=1}^d\gamma_{ii}+\sum
_{|i-j|=1}\frac{1}{\sqrt{1+ (\gamma_{ij} )^2}} \Biggr) + f(t,x),
\nonumber
\\[-8pt]
\\[-8pt]
g(x)&:=& \SIN(T,x),
\nonumber
\end{eqnarray}
and $f$ is chosen so that $u:= \SIN$ is the true solution of the PDE.
\end{eg}

In this case one may check straightforwardly that
\begin{eqnarray*}
&&[G_\gamma]_{ii}=3\quad\mbox{and}\quad \bigl[G_{\gamma}(t,x,y,z,
\gamma ) \bigr]_{ij}=-\frac{\gamma_{ij}}{(1+\gamma_{ij}^2)^{3/2}},\qquad |i-j|=1.
\end{eqnarray*}
When $d=10$, this example is out of the scope of our monotonicity
Assumption~\ref{assum-mon}, even with\vspace*{2pt} our choice of $p$ and $\sigma_0$:
$p:=\min ({1\over 3}$, ${1\over
 (1+d*(5-1))} )=1/41$, \mbox{$\sigma_0=I_d$}.
However, if we test it using $T=0.2$, $x_0=(1,2,\ldots,10)$, the
numerical results show that our scheme still converges to the true
solution, $\sin(55)= -0.999755$, as presented in Figure~\ref{d10tridiagonal}.

%
\begin{figure}
{\footnotesize\begin{tabular}{@{}ld{9.0}d{3.0}ccd{1.8}@{}}
\hline
$\bolds{n}$ & \multicolumn{1}{c}{$\bolds{L}$} & \multicolumn{1}{c}{$\bolds{K}$} & \multicolumn{1}{c}{\textbf{Avg(Ans.)}} & \textbf{Var(Avg.)} & \multicolumn{1}{c@{}}{\textbf{Cost (in seconds)}} \\
\hline
\phantom{00}2 & 2500 & 160 & $-$1.47362& $1.0\times10^{-5}$& 1.2\times10^{-2}\\
\phantom{00}5& 15{,}625 &64& $-$1.15004 & $1.7\times10^{-6} $&  1.4\times10^{-1}\\
\phantom{0}10&62{,}500&32& $-$1.06194 & $9.1 \times10^{-6} $&  1.0\times10^{0}\\
\phantom{0}20&250{,}000&16& $-$1.04519& $ 2.1\times10^{-6}$&  8.9\times10^{0}\\
\phantom{0}40&1{,}000{,}000& 8& $-$1.03326& $ 6.2 \times10^{-7}$&  7.1\times10^{1}\\
\phantom{0}80 &4{,}000{,}000&4& $-$1.03092& $ 5.8\times10^{-8} $& 5.9\times10^{2}\\
160& 16{,}000{,}000& 2& $-$1.01910& $ 3.0\times10^{-9} $& 1.4\times10^{4}\\
\hline
\end{tabular}}
\vspace*{12pt}

\includegraphics{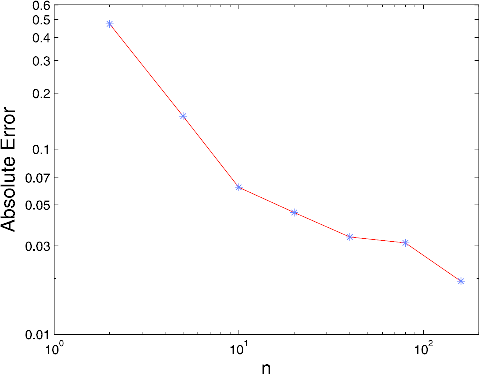}

\caption{A 10-dimensional example with tridiagonal generator in Example \protect\ref{eg-tri}.}\label{d10tridiagonal}
\end{figure}
We shall remark though that this example is computationally more
expensive than Example~\ref{12dimensionalExample} because here we need
to approximate $3d-2$ second derivatives.

\section*{Acknowledgments}\label{sec7}
The authors would like to thank Arash Fahim, Xiaolu Tan and two
anonymous referees for very helpful comments.




\printaddresses
\end{document}